\documentclass[11pt,a4paper,twoside]{article}
\RequirePackage{filecontents}
\usepackage[utf8]{inputenc}
\usepackage[UKenglish]{babel}
\usepackage[T1]{fontenc}
\usepackage{fancyhdr}
\usepackage[left=2cm,right=2cm,top=2cm,bottom=2cm]{geometry}
\usepackage{etex}
\usepackage[etex=true,export]{adjustbox}
\usepackage{float}
\usepackage{xkeyval}
\usepackage{mathtools}
\usepackage{amsmath}
\usepackage{amsthm}
\usepackage{amsfonts}
\usepackage{amssymb}
\usepackage{amssymb, amsfonts, amsmath, amsthm, times}
\usepackage{epsfig}
\usepackage{lipsum}
\usepackage{graphicx}
\usepackage{calc}
\usepackage{gensymb}
\usepackage{comment}
\usepackage{wrapfig}
\usepackage[sort, numbers]{natbib}
\usepackage{placeins}
\usepackage{hyperref}
\usepackage[dvipsnames]{xcolor}
\usepackage{ulem}
\usepackage{bbm}
\usepackage{authblk}

\theoremstyle{definition}
\newtheorem{mydef}{Definition}

\newtheorem{empl}{Example}
\newtheorem{conj}{Conjecture}
\theoremstyle{plain}
\newtheorem{theo}{Theorem}[section]
\newtheorem{lem}[theo]{Lemma}
\newtheorem{prop}[theo]{Proposition}
\newtheorem{cor}{Corollary}

\makeatletter
\newcommand\bigDiamond{\mathop{\mathpalette\bigDi@mond\relax}}
\newcommand\bigDi@mond[2]{%
  \vcenter{\hbox{\m@th
    \scalebox{\ifx#1\displaystyle 2\else1.2\fi}{$#1\Diamond$}%
  }}%
}
\newcommand\bigLozenge{\mathop{\mathpalette\bigL@zenge\relax}}
\newcommand\bigL@zenge[2]{%
  \vcenter{\hbox{\m@th
    \scalebox{\ifx#1\displaystyle 2\else1.2\fi}{$#1\blacklozenge$}%
  }}%
}
\makeatother

\binoppenalty=\maxdimen
\relpenalty=\maxdimen
\title{On left-orderability of involutory quandles of links}
\author[1]{Hamid Abchir}
\author[2]{Mohammed Sabak}
\affil[1]{Hassan II University. Fundamental and applied mathematics laboratory. Route d'El Jadida Km 8. B.P. 5366. Maarif. 20100 Casablanca. Morocco.
	{e-mail: hamid.abchir@univh2c.ma}}
\affil[2]{Hassan II University. Ain Chock Faculty of sciences. Route d'El Jadida Km 8. B.P. 5366. Maarif. 20100 Casablanca. Morocco. {e-mail: mohammed.sabak-etu@etu.univh2c.ma}}
\begin{document}
\maketitle
\begin{abstract}
  We give a non-left-orderability criterion for involutory quandles of non-split links. We use this criterion to show that the involutory quandle of any non-trivial alternating link is not left-orderable, thus improving Theorem 8.1. proven by Raundal \textit{et al.} (Proceedings of the Edinburgh Mathematical Society (2021 \textbf{64}), page 646). We also use the criterion to show that the involutory quandles of augmented alternating links are not left-orderable. We introduce a new family of links containing all non-alternating and quasi-alternating $3$-braid closures and show that their involutory quandles are not left-orderable. This leads us to conjecture that the involutory quandle of any quasi-alternating link is not left-orderable.
\end{abstract}

\section{Introduction}
Over the past few years, the study of orders on groups which arise naturally in topology has gained significance as a valuable tool for detecting some geometric properties of three-dimensional manifolds (\cite{boyer2005orderable}, \cite{boyer2013spaces}). The study of orderability was extended to semi-groups (\cite{sikora2004topology}) and, more generally, to magmas (\cite{dabkowska2007compactness},\cite{ha2018algorithmic},\cite{ha2018orders}). The orderability of quandles was introduced in \cite{bardakov2022zero} in the context of the study of zero-divisors in quandle rings. Recently, Hitesh Raundal, Mahender Singh and Manpreet Singh (\cite{raundal}) investigated the orderability of quandles with a focus on link quandles. They noticed that orderability of link quandles and that of the corresponding link groups behave quite differently. This observation arises from the study of the orderability of the link quandles of many families of links such as fibered prime knots, some non-trivial positive or negative links, and some torus links. In the same article, the authors discussed left-orderability of the involutory quandle of an alternating link $L$ denoted by $IQ(L)$. In particular, they showed the following result.

\begin{theo}[\cite{raundal}, Theorem 8.1]
Let $L$ be a non-trivial alternating link. If there exists a reduced alternating diagram $D$ of $L$ such that generators of the involutory quandle $IQ(L)$ of $L$ corresponding to arcs in $D$ are pairwise distinct, then $IQ(L)$ is not left-orderable.
\label{theo_0}
\end{theo}

As a consequence of the above result, it was shown that the involutory quandles of prime and alternating knots of prime determinants are not left-orderable (\cite{raundal}, Corollary 8.3). The same was shown also for prime and alternating non-trivial Montesinos links (\cite{raundal}, Corollary 8.2) and some alternating Turk's head knots (\cite{raundal}, Corollary 8.4).    
  
The main result of the present paper (Theorem \ref{theo_1}) gives a non-left-orderability criterion for involutory quandles of non-split links. In order to obtain this criterion we introduce what we call the \textit{coarse presentation} of the involutory quandle of a link (Definition \ref{mydef1}), which is derived from the usual presentation of the involutory quandle.

To define the coarse presentation of the involutory quandle, we were inspired by the \textit{coarse Brunner's presentation} introduced by Ito in \cite{ito2013non}. For a non-split link $L$, the coarse Brunner's presentation of the fundamental group $\pi_1 \left( \Sigma_2(L) \right)$ of the double branched cover $\Sigma_2(L)$ of $L$ is a structure derived from the \textit{Brunner's presentation} of $\pi_1 \left( \Sigma_2(L) \right)$, introduced by Brunner in \cite{brunner1997double}. The coarse Brunner's presentation was used by Ito to give a non-left-orderability criterion of the group $\pi_1 \left( \Sigma_2(L)  \right)$ (\cite{ito2013non}, Theorem 3.11). Note also that the group $\pi_1 \left( \Sigma_2(L)  \right)$ is actually closely related to $IQ(L)$. In fact, in his Ph.D. thesis (\cite{winker1984quandles}, Theorem 5.2), Winker established that the group $\pi_1 \left( \Sigma_2(L)  \right)$ is isomorphic to a subgroup of a quotient of the enveloping group of $IQ(L)$.

The criterion given in Theorem \ref{theo_1} allowed us to show that Theorem \ref{theo_0} extends to any non-trivial alternating link (Theorem \ref{theo_2}). This criterion is also used to show that the involutory quandle of an augmented alternating link is not left-orderable (Theorem \ref{theo_3}). Finally, we introduce a new family of links containing all non-alternating and quasi-alternating $3$-braid closures and show that their involutory quandles are also not left-orderable (Theorem \ref{theo_4}). Hence, by extending arguments introduced in \cite{champanerkar2009twisting} and \cite{abchir2020generating}, we get infinite families of non-alternating quasi-alternating links that have non-left involutory quandles as explained in the last section. This enables us to conjecture that this will be the case for all quasi-alternating links.

The paper is organized as follows. The second section is devoted to some preliminaries. We recall the structure of quandles and the notion of orderability on them. We give some examples with a focus on link quandles and involutory quandles of links. In section three, we prove Proposition \ref{prop1} which enables us to introduce the coarse presentation of involutory quandles of links. We end this section by showing the main theorem of this paper which provides the non-left-orderability criterion for involutory quandles of links. The fourth section contains some results derived from the main theorem which state that the involutory quandles of alternating links, augmented alternating links, and quasi-alternating $3$-braid closures are all not left-orderable. We end by a section devoted to a discussion leading to the statement of a conjecture.

\section{Preliminaries}
\paragraph*{Quandles.}
A \textit{quandle} is a set $Q$ equipped with a binary operation $*$ that satisfies the following axioms:
\begin{itemize}
\item[A1.] \textit{Idempotency}: $x*x=x$ for all $x \in Q$.
\item[A2.] \textit{Right-invertibility}: For each $x,y \in Q$, there exists a unique element $z \in Q$ such that $x=z*y$.
\item[A3.] \textit{Self-distributivity}: $(x*y)*z = (x*z)*(y*z)$ for all $x,y,z \in Q$.
\end{itemize}

Note that the axiom A2 is equivalent to the existence of a dual operation $*^{-1}$ on $Q$ such that $$(x*y)*^{-1}y = (x*^{-1}y)*y=x \text{ for all } x,y \in Q.$$

The quandle $Q$ is said to be \textit{involutory} if $* = *^{-1}$. So axiom A2 becomes equivalent to $(x*y)*y=x$ for all $x,y \in Q$.  

Let $(Q_1,*_{1})$ and $(Q_2,*_{2})$ be two quandles. A \textit{quandle homomorphism} is a map $f:Q_1 \longrightarrow Q_2$ satisfying $f\left( x*_{1}y \right)=f\left( x \right)*_{2}f\left( y \right)$, for every two elements $x,y \in Q_1$. If the quandle homomorphism $f$ is bijective, then it is called a \textit{quandle isomorphism}, in which case the quandles $Q_1$ and $Q_2$ are said to be \textit{isomorphic}.  

\paragraph*{Examples.}
\begin{itemize}
\item[•] \textbf{\textit{The trivial quandle.}} A quandle $Q$ is \textit{trivial} if $y*x=x$ for every elements $x,y \in Q$. The trivial quandle on a set $X$ is defined as the quandle whose underlying set is $X$ and whose operation satisfies $y*x=x$ for all elements $x,y \in X$. 
\item[•] \textbf{\textit{The conjugation quandle of a group.}} If $G$ is a group and $n \in \mathbb{N}$, then the binary relation $x*y:=y^{-n} x y^n$ turns $G$ into a quandle ${\rm Conj}_n(G)$ called the $n$\textit{-conjugation quandle} of $G$. The quandle ${\rm Conj}_1(G)$ is simply called the \textit{conjugation quandle} of $G$ and is denoted by ${\rm Conj}(G)$. 
\item[•] \textbf{\textit{The conjugation quandle of the enveloping group.}} Let $Q$ be a quandle. The \textit{enveloping group} of $Q$, denoted ${\rm Env}(Q)$, is the group with generators $\left\lbrace \tilde{x} | x \in Q \right\rbrace$ and relations $\widetilde{x*y}=\tilde{y}^{-1} \tilde{x} \tilde{y}$ for all $x,y \in Q$. The conjugation quandle of ${\rm Env}(Q)$ provides a natural quandle homomorphism
  $$
  \begin{array}[c]{cccc}\label{eta}
  \eta: & Q & \longrightarrow & {\rm Conj}({\rm Env}(Q))\\ & x & \longmapsto & \tilde{x}
\end{array}
$$
which is not injective in general.
\item[•] \textbf{\textit{The core quandle of a group.}} If $G$ is a group, then the binary relation $x*y:=yx^{-1}y$ turns $G$ into a quandle called the \textit{core quandle} of $G$ and denoted by ${\rm Core}(G)$. For any group $G$, the quandle ${\rm Core}(G)$ is involutory. If $G$ is cyclic of order $n$, then ${\rm Core}(G)$ is called the $n$-\textit{dihedral quandle}, and it is denoted by $R_n$. 
\item[•] \textbf{\textit{Quandles given by a presentation.}} It is known that any quandle $Q$ can be given by a presentation $\langle S \big| R \rangle$ where $S$ is a generating set of $Q$ and $R$ is a set of relations (\cite{bardakov2020embeddings}).

  Note that the quandle operation is, in general, not associative. To better highlight this defect, we adopt the \textit{exponential notation} due to Fenn and Rourke (\cite{fenn1992racks}) defined by:
    $$x^y:=x*y,\,\,{\rm and}\,\,\,x^{y^{-1}}:=x*^{-1}y.$$
    More precisely this notation allows brackets to be dispensed with, because there are standard conventions for association with exponents. In particular
    $$x^{yz}=(x^y)^z=(x*y)*z,\,\, {\rm whereas}\,\,\, x^{y^z}=x^{(y^z)}=x*(y*z).$$
    However, to make calculations simpler, we will often need to mix the exponential notation with the standard symbol of the binary operation $*$.\\

    The elements of a quandle which are written in the form $$\left( \left( (x_1 *^{\pm 1} x_2 ) *^{\pm 1} x_3 \right) \hdots \right) *^{\pm 1} x_n = x_1^{x_2^{\pm 1} x_3^{\pm 1} \hdots x_n^{\pm 1}}$$
are called \textit{left-associative products}. Equivalently, Left-associative products of a quandle $Q$ are the elements of the form $x^u$ where $x$ is an element of $Q$ and $u$ is a word in the free group on $Q$. We denote $u^{-1}$ by $\bar{u}$. The following lemma shows that left-associative products are closed under the quandle operation. 

\begin{lem}[\citep{fenn1992racks}]
If $x^u$ and $y^v$ are elements of a quandle, then 
$$(x^u)^{(y^v)}=x^{u\bar{v}yv} \text{ and }  (x^u)^{\overline{(y^v)}}=x^{u\bar{v}\bar{y}v} .$$
\label{lem_0}
\end{lem}

From the previous lemma we conclude that if $Q = \langle S \big| R \rangle$, then the elements of $Q$ can be represented as equivalent classes of left-associative products $x^w$ where $x$ is a generator in $S$ and $w$ is a word in the free group on $S$. If $Q$ is involutory, then by the self-invertibility axiom, any element of $S$ has order $2$ in the free group, so $\bar{x}=x$ for any generator $x \in S$. Furthermore, the inverse of a word $w=x_1x_2 \hdots x_n$ in the free group on $S$ is $\bar{w}=x_n \hdots x_2x_1$. 

If $Q$ is a quandle given by a presentation $\langle S \big| R \rangle$, then the group ${\rm Env}(Q)$ can be given a simple presentation. In fact, it was shown by Raundal \textit{et al.} (\citep{raundal}, Theorem 2.2) that ${\rm Env}(Q)$ can be presented as $\langle \tilde{x}, x \in S \big| \tilde{R} \rangle$, such that $\tilde{R}$ consists of relations obtained from the relations in $R$ by replacing an expression $x*y$ by $\tilde{y}^{-1} \tilde{x} \tilde{y}$ and replacing an expression $x*^{-1}y$ by $\tilde{y} \tilde{x} \tilde{y}^{-1}$. 

\item[•] \textbf{\textit{Free quandles.}} Let $S$ be a set. The \textit{free quandle} on $S$, denoted $FQ(S)$, is the quandle whose elements are the left-associative products $x^w$ where $x$ is a generator in $S$ and $w$ is a word in the free group on $S$, and the quandle operation on $FQ(S)$ is given by Lemma \ref{lem_0}.

\item[•] \textbf{\textit{The link quandle and the involutory quandle of a link.}} Let $D$ be an oriented diagram of an oriented link $L$. We denote by $A(D)$ and $C(D)$ respectively the set of arcs and the set of crossings of $D$. The \textit{fundamental quandle} or \textit{link quandle} $Q(L)$ of $L$, which is an invariant of oriented links introduced independently by Joyce (\cite{joyce1982classifying},\cite{joyce1979algebraic}) and Matveev (\cite{matveev1984distributive}), is the quandle given by the following presentation
 $$\langle x: x\in A(D)|r_c:c\in C(D)\rangle,$$
 where, for each crossing $c$, the relation $r_c$ is obtained as in Figure \ref{fig1}. In fact, we consider the set of equivalence classes of quandle words in $A(D)$ modulo the equivalence relation generated by axioms $A_1$, $A_2$ and $A_3$, i.e. the \textit{free quandle} on $A(D)$, itself quotiented by the equivalence relation generated by the crossing relations (\cite{elhamdadi2015quandles}).

Note that $Q(L)$ depends on the orientation of $L$. We can however derive a quandle which does not depend on the orientation of $L$ from $Q(L)$. We define the \textit{involutory quandle} of $L$, denoted $IQ(L)$, as the quandle obtained from $Q(L)$ by adding the relation $x^{y^2}=x$ for every two generators $x$ and $y$ of $Q(L)$. In this setting, the relations $z=y^{x}$ and $y=z^{x}$ become equivalent, which means that $IQ(L)$ does not depend on the orientation of $L$. Consequently, the involutory quandle is an invariant of unoriented links.

\begin{figure}[H]
\centering
\includegraphics[width=0.4\linewidth]{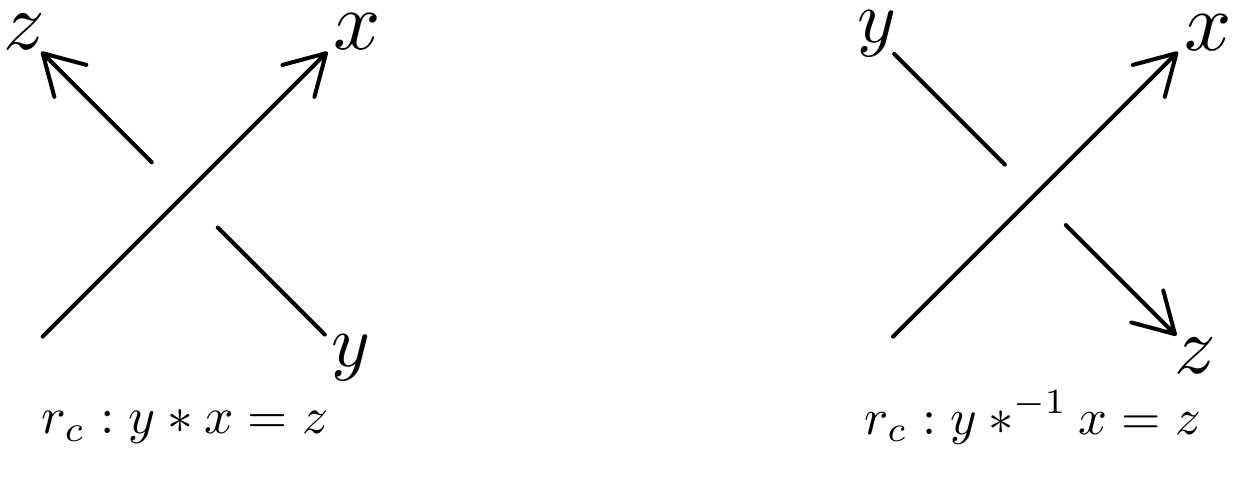}   
\caption{Relation of the fundamental quandle at a crossing $c$.}
\label{fig1}
\end{figure} 
\end{itemize}

\paragraph*{Notation.} Let $A$ be any subset of the set $A(D)$ of arcs of a link diagram $D$ representing a link $L$. We denote by $A_{IQ}$ the set of generators of $IQ(L)$ corresponding to the elements of $A$. Note that different elements of $A$ do not necessarily correspond to different generators of $IQ(L)$. If $x$ is an element of $A$, we will abuse notation and denote also by $x$ the corresponding generator in $A_{IQ}$. For two distinct elements $x$ and $y$ of $A$ corresponding to the same generator of $IQ(L)$, the equality $x=y$ is understood to be inside $A_{IQ}$.

\paragraph{Orderability of quandles.}
A \textit{left-order} on a quandle $Q$ is a strict linear order $<$ on $Q$ such that if $x < y$, then $z*x < z*y$ for every elements $x,y,z \in Q$. A quandle on which there is a left-order is said to be \textit{left-orderable}. Similarly, a \textit{right-order} on a quandle $Q$ is a strict linear order on $Q$ such that if $x<y$, then $x*z<y*z$ for all elements $x,y,z\in Q$. A quandle on which there is a right-order is said to be \textit{right-orderable}. A \textit{bi-order} on $Q$ is a strict linear order on $Q$ that is simultaneously a left and a right-order. A quandle on which there is a bi-order is said to be \textit{bi-orderable}. 

For example, a trivial quandle is not left-orderable but a trivial quandle with more than one element is right-orderable (\cite{dabkowska2007compactness}). Any non-trivial left or right-orderable quandle is infinite (\cite{raundal}, Proposition 2.5). So, finite quandles are neither right nor left-orderable. On the other hand, many left or right-orderable quandles can be generated from some groups as shown by the following results: If the group $G$ is bi-orderable, then the quandle ${\rm Conj}(G)$ is right-orderable (\cite{dabkowska2007compactness}, Proposition 7) whilst the quandle ${\rm Core}(G)$ is left-orderable (\cite{bardakov2022zero}, Proposition 3.4). On the other hand, if $G$ is any non-trivial group, then ${\rm Conj}(G)$ is not left-orderable and ${\rm Core}(G)$ is not left-orderable (\cite{bardakov2022zero}, Corollaries 3.8 and 3.9).

Let $Q$ be a quandle generated by a set $S$. The quandle homomorphism $\eta$ defined on Page \ref{eta} was used to provide the following sufficient condition for a quandle to be not left-orderable.

\begin{prop}[\cite{raundal}, Proposition 5.3]
Let $Q$ be a quandle generated by a set $S$ such that the map $\eta$ is injective. If there exist two distinct commuting elements in ${\rm Env}(Q)$ that are not inverses of each other and that are conjugates of elements from $\eta(S)^{\pm 1}$, then $Q$ is not left-orderable.  
\end{prop}

A notable fact about involutory quandles is that they are not right-orderable (\cite{bardakov2022zero}, Proposition 3.7). It is known that free involutory quandles are left-orderable (\cite{dhanwani2021dehn}, Proposition 4.4). So, the involutory quandle of a trivial link with more than one component is left-orderable. Significant examples of non-orderable quandles are the involutory quandles of some alternating links as mentionned in Theorem \ref{theo_0} cited in the introduction. As a consequence of that result, it was shown that the involutory quandles of prime and alternating knots with prime determinants are not left-orderable (\cite{raundal}, Corollary 8.3). The same was shown also for prime and alternating non-trivial Montesinos links (\cite{raundal}, Corollary 8.2) and for some alternating Turk's head knots (\cite{raundal}, Corollary 8.4).

\paragraph{Rational tangles.} We call a \textit{tangle} $T$ any pair $(B,A)$, where $B$ is a $3$-ball in the $3$-sphere $S^3$ and $A$ is a properly embedded $1$-manifold in $B$ which meets the boundary $\partial B$ of $B$ in four distinct points, together with an identification of the pair $(\partial B, \partial B \cap A)$ with $\left( S^2,\left\lbrace NE,NW,SW,SE \right\rbrace \right)$ where $S^2$ is the $2$-sphere with symbols NE, NW, SW and SE referring to the compass directions as in Figure \ref{Compass}. Two tangles $(B,A)$ and $(B,A^{'})$ are \textit{equivalent} if there is an isotopy of $B$ that is the identity on $\partial B$ and that takes $A$ to $A^{'}$.
\begin{figure}[H]
\centering
\includegraphics[width=0.25\linewidth]{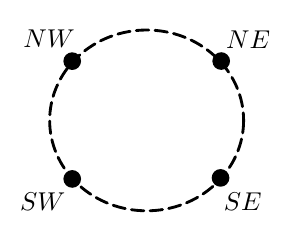}   
\caption{Labelling of boundary points.}
\label{Compass}
\end{figure} 
Let $T=(B,A)$ be a tangle. We assume that the four endpoints lie in the great circle of the boundary sphere which joins the two poles. That great circle bounds a disk $\Delta$ in $B$. We consider a regular projection of $B$ on $\Delta$. The image of $A$ by that projection in which the height information is added at each of the double points is called a \textit{tangle diagram} of $T$. Two tangle diagrams are \textit{equivalent} if they are related by a finite sequence of planar isotopies and Reidemeister moves in the interior of the \textit{projection disk} $\Delta$. Note that two tangles are equivalent if and only if they have equivalent diagrams. Depending on the context we will denote by $T$ the tangle or its projection. We will denote by $A(T)$ and $C(T)$ respectively the set of arcs and the set of crossings of a tangle diagram $T$.

A tangle diagram $T$ is said to be \textit{disconnected} if either there exists a simple closed curve embedded in the projection disk which does not meet $T$ but encircles a part of it, or there exists a simple arc properly embedded in the projection disk which does not meet $T$ and splits the projection disk into two disks each one containing a part of $T$. A tangle diagram is \textit{connected} if it is not disconnected. A tangle diagram is \textit{reduced} if its number of crossings can not be reduced by any tangle equivalence. In the remaining of the paper, all tangle diagrams are assumed to be connected and reduced.

A tangle diagram $T$ provides two link diagrams: the \textit{numerator} of $T$, denoted by $n(T)$, which is obtained by joining with simple arcs the two upper endpoints $(NW,NE)$ and the two lower endpoints $(SW,SE)$ of $T$, and the \textit{denominator} of $T$, denoted by $d(T)$, which is obtained by joining with simple arcs each pair of the corresponding top and bottom endpoints $(NW,SW)$ and $(NE,SE)$ of $T$ (see Figure \ref{fig2}). We denote by $N(T)$ and $D(T)$ respectively the links represented by the diagrams $n(T)$ and $d(T)$. 

\begin{figure}[H]
\centering
\includegraphics[width=0.3\linewidth]{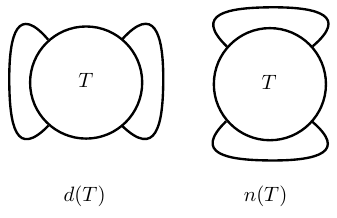}   
\caption{The denominator and the numerator of a tangle diagram $T$.}
\label{fig2}
\end{figure}

A tangle diagram is \textit{alternating} if the ''over'' or ''under'' nature of the crossings alternates as one moves along any arc of the tangle diagram. A tangle is said to be \textit{alternating} if it admits an alternating diagram. 

Let $T$ be an alternating tangle diagram. Consider the arc of $T$ which have $NW$ as an endpoint. Suppose that when we move along that arc starting at $NW$ we pass below at the first encountered crossing. Then the arc of $T$ which ends at the point $SE$ will also pass below at the last encountered crossing before reaching $SE$ and the arc of $T$ which starts at $NE$ will pass over at the first encountered crossing. It is easy to see that the arcs of $T$ coming from diametrically opposite endpoints both pass over or below at the first encountered crossing. That remark enables us to distinguish two types of alternating tangle diagrams which we call \textit{type 1} tangles and \textit{type 2} tangles as shown in the Fig. \ref{fig3}. 

\begin{figure}[H]
\centering
\includegraphics[width=0.5\linewidth]{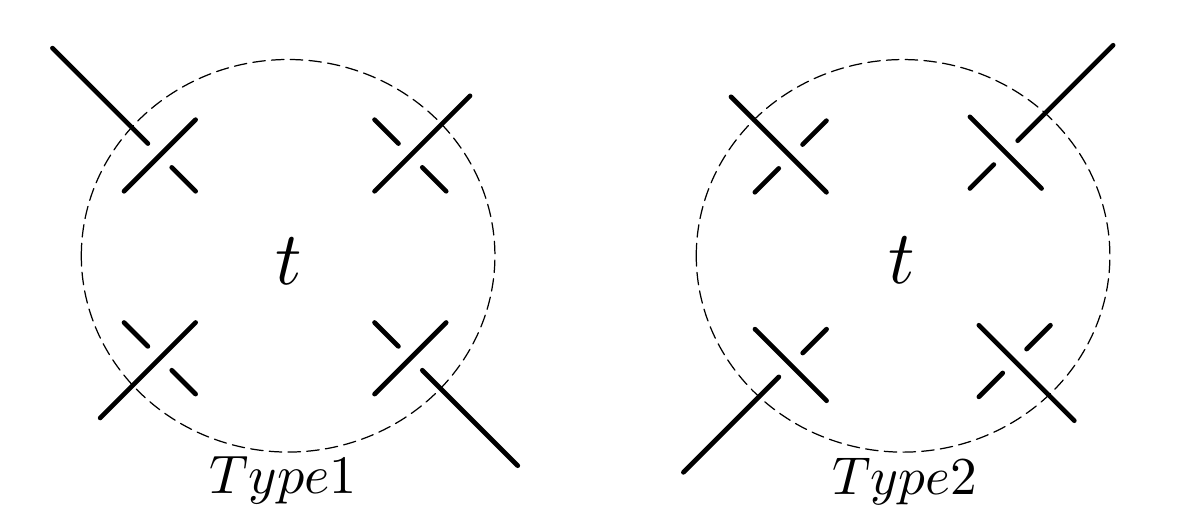}   
\caption{Type 1 and Type 2 alternating tangle diagrams.}
\label{fig3}
\end{figure}

We adopt the notations used for tangles by Goldman and Kauffman in \cite{goldman} and Kauffman and Lambroupoulou in \cite{kauffman2}. The \textit{mirror image} of a tangle $T$, denoted $-T$, is obtained from $T$ by inverting the height information at each crossing. If we rotate the tangle $-T$ by a $90^{ \circ }$ angle in the counter clockwise (respt. clockwise) direction, we get a tangle denoted $\frac{1}{T_{cc}}$ (respt. $\frac{1}{T_c}$). In Figure \ref{fig4}, we recall some operations defined on tangles.

\begin{figure}[H]
\centering
\includegraphics[width=0.5\linewidth]{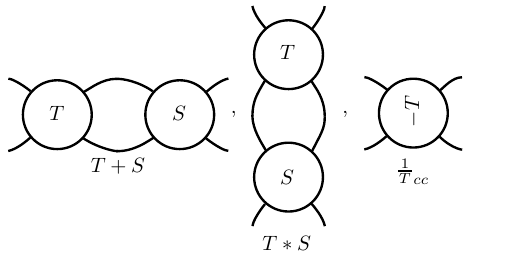}   
\caption{Some operations on tangle diagrams.}
\label{fig4}
\end{figure}

A $180^{\circ}$ rotation of a tangle diagram $T$ in the horizontal (respt. vertical) axis is called \textit{horizontal Flip} (respt. \textit{vertical Flip}) and will be denoted by $T_h$ (respt. $T_v$). That is the tangle diagram obtained by rotating the ball containing $T$ in space around the horizontal (respt. vertical) axis as shown in the left of Figure \ref{fig5} and then project the new tangle by the same projection function as that used to get $T$. Note that if $T$ is an alternating tangle diagram, then $T_h$ and $T_v$ are also alternating. A \textit{flype} is an isotopy of tangles that is depicted in the right of Figure \ref{fig5}.

\begin{figure}[H]
\centering
\includegraphics[width=0.7\linewidth]{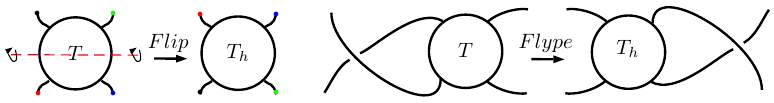}
\caption{the Flip move (left), and the flype move (right).}
\label{fig5}
\end{figure}

A \textit{rational tangle} is a tangle $(B,A)$ such that the pair $(B,A)$ is homeomorphic to $(D^2 \times [0,1] , \left\lbrace x,y \right\rbrace \times [0,1]) $, where $D^2$ is the unit disk and $x$ and $y$ are points in the interior of the disk $D^2$. The rational tangle diagrams $0$, $\pm1$, $\infty$ are shown in Figure \ref{fig6}.

\begin{figure}[H]
\centering
\includegraphics[width=0.15\linewidth]{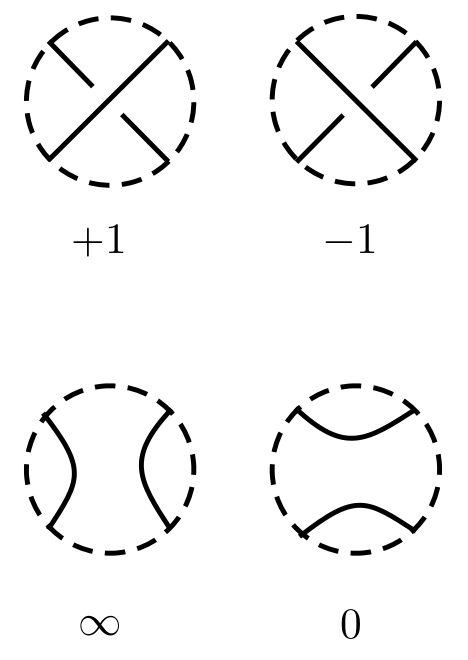}  
\caption{The rational tangle diagrams $0$, $\pm1$, and $\infty$ }
\label{fig6}
\end{figure}

The sum of $n$ copies of the tangle diagram $1$ or of $n$ copies of the tangle diagram $-1$ are also respectively denoted $n$ and $-n$.
If $t$ is a rational tangle diagram then $\frac{1}{t_c}$ and $\frac{1}{t_{cc}}$ are equivalent and both represent the \textit{inversion} of $t$ denoted by $\frac{1}{t}$. The tangle diagrams $n$, $-n$, $\dfrac{1}{n}$, and $-\dfrac{1}{n}$ are called \textit{elementary tangle diagrams}.

Let $t$ be a rational tangle diagram and $p,q \in \mathbb{Z}$, we have the following equivalences that can easily be proven using flype moves:
$$  p+t+q = t+p+q \text{  ,  } \frac{1}{p} * t * \frac{1}{q} = t * \frac{1}{p+q}.$$
$$ t * \frac{1}{p} = \frac{1}{p+\frac{1}{t}} \text{  ,  } \frac{1}{p} * t = \frac{1}{\frac{1}{t}+p}. $$

Using the above notations and equivalences, one can naturally associate a tangle diagram to any continued fraction
$\left[ a_1,\hdots,a_{n} \right]:=\displaystyle a_1 + \frac{1}{ a_2  + \frac{1}{  \ddots +\frac{1}{ a_{n-1}  + \frac{1}{ a_n } }}},\,\ a_i\in\mathbb{Z},$ as shown in Figure \ref{fig7}. It is clear that if all the integers $a_i$ are of the same sign, then the associated tangle diagram is alternating.

\begin{figure}[H]
\centering
\includegraphics[width=0.75\linewidth]{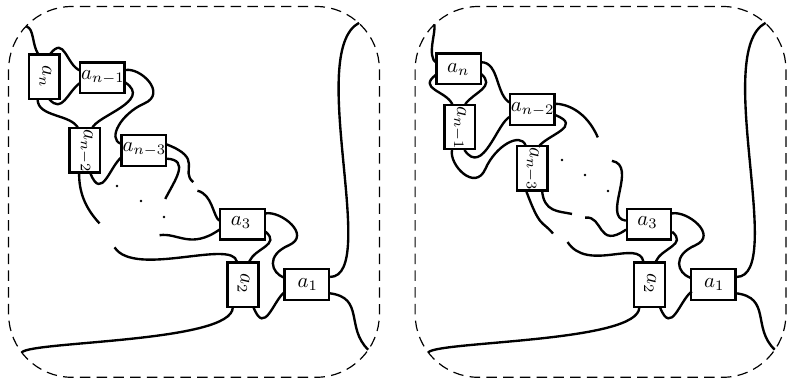}  
\caption{The rational tangle diagram associated to the continued fraction $\left[ a_1,\hdots,a_{n} \right]$ according to whether $n$ is even (left) or odd (right).}
\label{fig7}
\end{figure}

In \cite{conway}, Conway showed that for any rational tangle $t$, there exists a unique continued fraction $\left[ a_1  , \hdots ,  a_n \right]$ where all the integers $a_i$ are of the same sign, called the \textit{fraction} of $t$, such that the associated alternating tangle diagram depicted in Figure \ref{fig7} represents $t$. Furthermore, Conway showed that two rational tangles are equivalent if and only if they have the same fraction. Thus, any rational tangle $t$ can be represented by a unique rational number $\left[ a_1  , \hdots ,  a_n \right] = \frac{p}{q}$ where all the $a_i$ are of the same sign. If $a_1$ is non-zero, then the number $n$ of elementary subtangles of $t$ is called the \textit{length} of $t$. Since, $\left[ a_1 , \hdots , a_n, 1 \right]$ and $\left[ a_1, \hdots , a_n + 1 \right]$ are equal as continued fractions, then one can assume without loss of generality that the length of a rational tangle is always odd. We make this assumption throughout the rest of the paper.

The \textit{standard diagram} of a rational tangle $t$ will be the connected reduced alternating tangle diagram naturally associated to the fraction of $t$ described above. In what follows a rational tangle diagram will mean the standard one.

\section{A non-left-orderability criterion for involutory quandles of links}
In this section, we show the main theorem of this paper which gives a non-left-orderability criterion for involutory quandles of links. We start by adapting the notion of tangle-strand decomposition introduced by Ito to our context. Then, we state and show a fundamental remark (Proposition \ref{prop1}) which allows us to define the coarse presentation of the involutory quandle of a non-split link.

\subsection{Tangle-strand decomposition.}

\begin{wrapfigure}[11]{r}{0.3\textwidth}	
\vspace{-40pt}
\begin{center}
\raisebox{0pt}[\height+\baselineskip]{\includegraphics[width=0.5\linewidth]{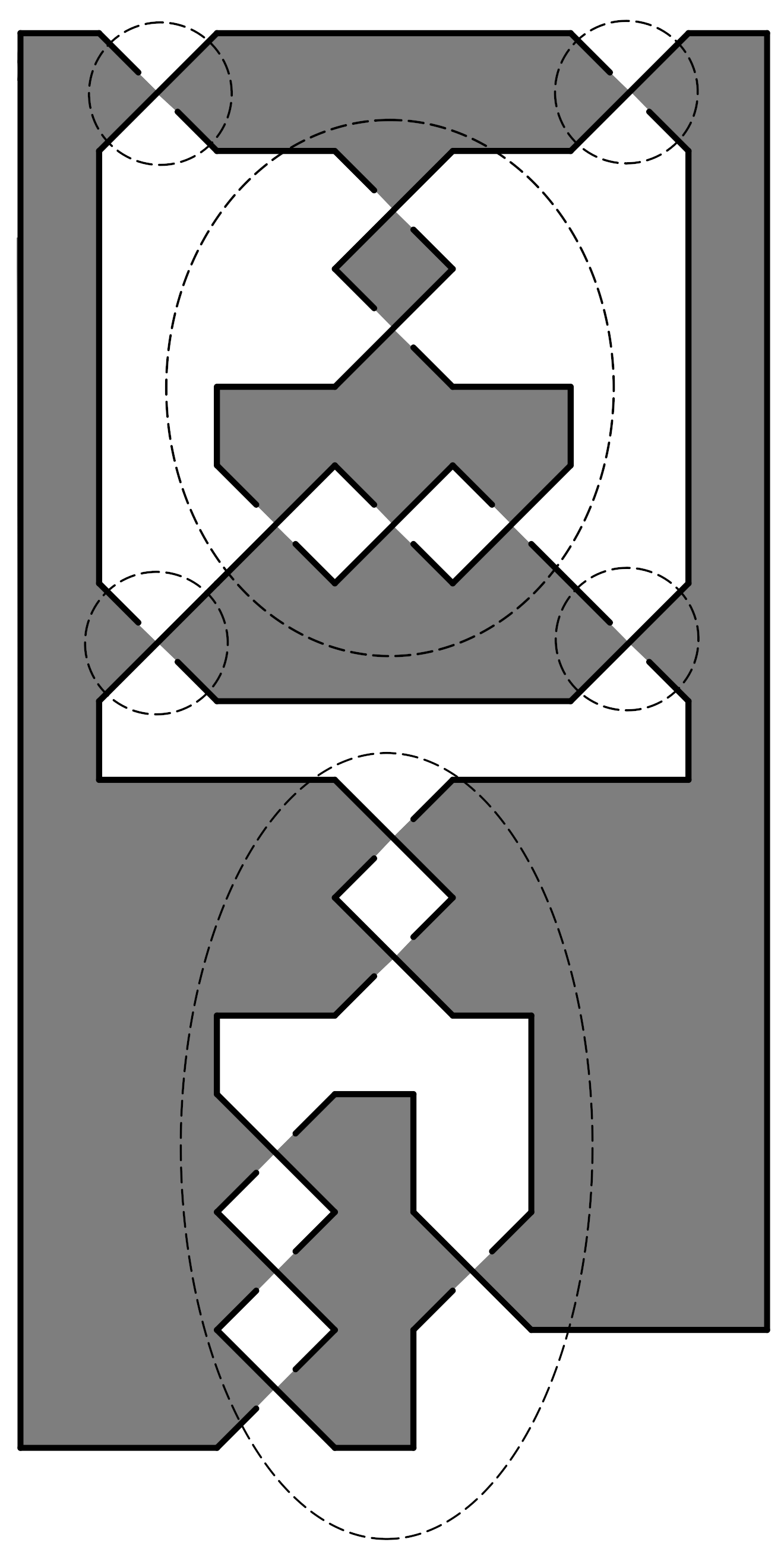}}
\end{center}
\end{wrapfigure}

Let $D$ be a diagram of a non-split link $L$. We consider a checkerboard coloring of $D$ with the convention that the unbounded region is not colored. The colored regions define a compact surface, possibly non-orientable, whose boundary is the link $L$. We call the obtained surface a \textit{checkerboard surface}. The link diagram $D$ can be decomposed into embedded rational tangles attached together with a set of strands as was described by Ito in \cite{ito2013non}. Such a decomposition of $D$ induces a decomposition of its checkerboard surface into a set of disks called \textit{disk-parts} of $D$ and subsurfaces corresponding to tangles called \textit{tangle parts} of $D$. The obtained decomposition is called a \textit{tangle-strand decomposition} of $D$. The figure at the right depicts an example of a tangle-strand decomposition of a link diagram where every tangle part is specified by a dashed circle around it.

Note that our definition of the tangle-strand decomposition is slightly different from the one given in \cite{ito2013non} where the author considered the tangle parts to be subsurfaces of the checkerboard surface corresponding to \textit{algebraic tangles} which are obtained from rational tangles by finite sequences of the tangle operations $+$ and $*$ illustrated in Figure \ref{fig4}. 

Let $A$ be a tangle part of $D$ and let $\Delta$ be the plane projection disk of the corresponding tangle. We consider the two arcs which constitute the intersection of $\partial\Delta$ with $A$. To distinguish the isotopy class of the tangle we will use, we agree that the endpoints of these two arcs will be labelled in the following way: one of them connects $NW$ to $SW$ while the other connects $NE$ to $SE$ (see the left of Figure \ref{fig8}). Since rational tangles are invariant by Flip moves, then there exists a unique rational tangle $t$ corresponding to this labelling. We abuse the notation and denote the tangle part $A$ also by $t$. Once the isotopy class of a tangle $t$ is chosen, we will fix a labelling for two special arcs lying in the corresponding tangle part. If $t$ is positive, in which case it is of type 1 as an alternating tangle, then we label by $X_t$ and $Y_t$ the two arcs passing through the end points $NW$ and $SE$ respectively. If $t$ is negative, in which case it is of type 2, then we label by $X_t$ and $Y_t$ the two arcs which pass through the end points $NE$ and $SW$ respectively. The situation is depicted in the left of Figure \ref{fig8}.

\begin{figure}[H]
\centering
\includegraphics[width=0.7\linewidth]{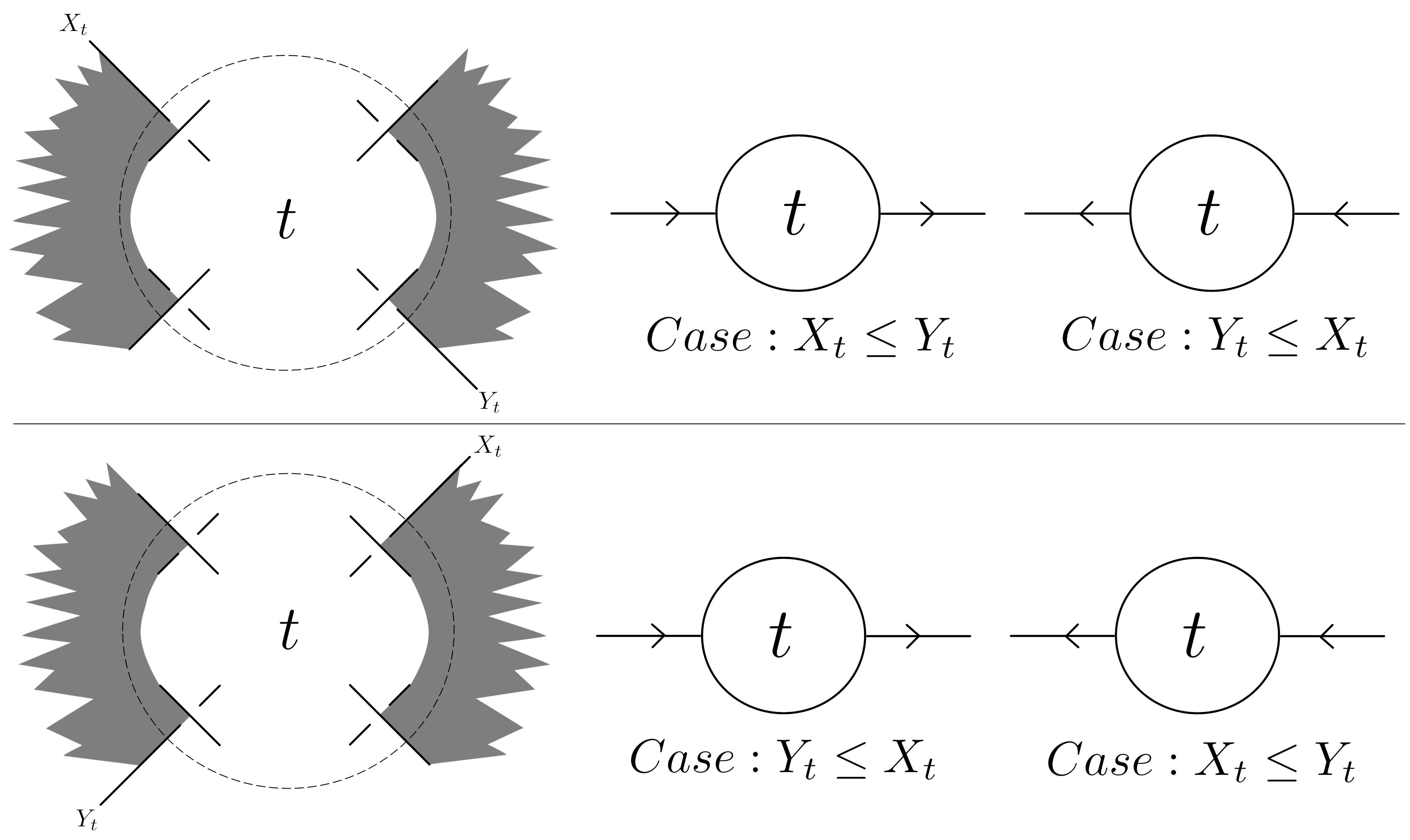}  
\caption{The arcs $X_t$ and $Y_t$ according to the type of the rational tangle (left) and the rules of the orientation of the graph $\Gamma_{D,<}$ (right).}
\label{fig8}
\end{figure}

Let $L$ be a non-split link with a connected diagram $D$. To the tangle-strand decomposition of the diagram $D$, we associate a planar graph $\Gamma_D$ called the \textit{coarse decomposition graph} of $D$ in the following way: The vertex of $\Gamma_D$ is a disk part of the tangle-strand decomposition. To each tangle part we assign an edge connecting the vertices that correspond to the disks connected by the considered tangle part. By a \textit{region} of the diagram $D$ we mean a region of the graph $\Gamma_D$. Each region $R$ of $D$ is surrounded by a set of tangle parts which we denote by $\partial R$. In what follows, we will use the set $A_{IQ}(R):= \displaystyle \bigcup_{t \in \partial R} A_{IQ}(t)$.

The \textit{dual} of $\Gamma_D$ is a planar graph denoted by $\left( \Gamma_D \right)^*$ which is obtained from $\Gamma_D$ by placing a vertex in each region of $\Gamma_D$ (including the unbounded one) and by putting an edge for every pair of regions in $\Gamma_D$ that are adjacent to each other through an edge of $\Gamma_D$, and a self-loop in the case where the same region is present on both sides of an edge. 

Note that our definition of the graph $\Gamma_D$ is different from the coarse decomposition graph defined in \cite{ito2013non} where the author assigned an orientation to $\Gamma_D$. We will also define an oriented version of $\Gamma_D$ in the case where $IQ(L)$ is left-orderable.

Assume that $<$ is a left-order on $IQ(L)$. We choose a way to orient the edges of the coarse decomposition graph $\Gamma_D$ which depicts how the generators involved in each tangle part are ordered with respect to $<$ (see the right of Figure \ref{fig8}). This provides an oriented planar graph denoted $\Gamma_{D,<}$. If $t$ is a tangle part of $D$ such that $X_t = Y_t$, then the corresponding edge of $\Gamma_{D,<}$ will be oriented arbitrarily. So, the order of generators in $t$ will not depend on the orientation of such edges.

\subsection{Coarse presentation of the involutory quandle of a non-split link}
Let $Q$ be a left-orderable involutory quandle. Let $LO(Q)$ denote the set of left-orders on $Q$. For each $< \in LO(Q)$ and $x \in Q$, we define 

$$\left( \infty,x \right]_<:=\left\lbrace z \in Q | z \leq x\right\rbrace , \left[ x,\infty \right)_<:=\left\lbrace z \in Q | x\leq z \right\rbrace.$$
For any $x$ and $y$ of $Q$, we define the following subset of $LO(Q)$:
$$ LO_{(x,y)}(Q) := \left\lbrace < \in LO(Q) | x \leq y \right\rbrace.$$
For any left-order $< \in LO_{(x,y)}(Q)$, define the following subset of $Q$:
$$ \left[ x,y \right]_< := \left\lbrace z \in Q | x \leq z \leq y \right\rbrace .$$
For any $x$ and $y$ in $Q$, we define the following subsets of $Q$: 

$$ \left[ x,y \right] := \bigcap_{< \in LO_{(x,y)}(Q)} \left[ x,y \right]_< .$$

$$ \left\lbrace | x,y | \right\rbrace := \left[ x,y \right] \cap \left[ y,x \right].$$

Note that $\left\lbrace | x,y | \right\rbrace = \left\lbrace | y,x | \right\rbrace$, and $\left\lbrace | x,x | \right\rbrace = \left\lbrace x \right\rbrace$ for all $x,y \in Q$. 

Let $L$ be a non-split link with a diagram $D$ and let $t$ be a tangle part of $D$ of length $n=2k+1$. We will enumerate the elementary subtangles $t_1, \hdots, t_n$ of $t$ according to the type of $t$ as in Figure \ref{fig9}. For each elementary subtangle $t_i$, we will label the elements of the set $A(t_i)$ as $\left\lbrace x_{(i,0)}, \hdots, x_{(i,c_i+1)} \right\rbrace$, where $c_i$ is the number of crossings in $t_i$, as in Figure \ref{fig10}. 

\begin{figure}[H]
\centering
\includegraphics[width=0.95\linewidth]{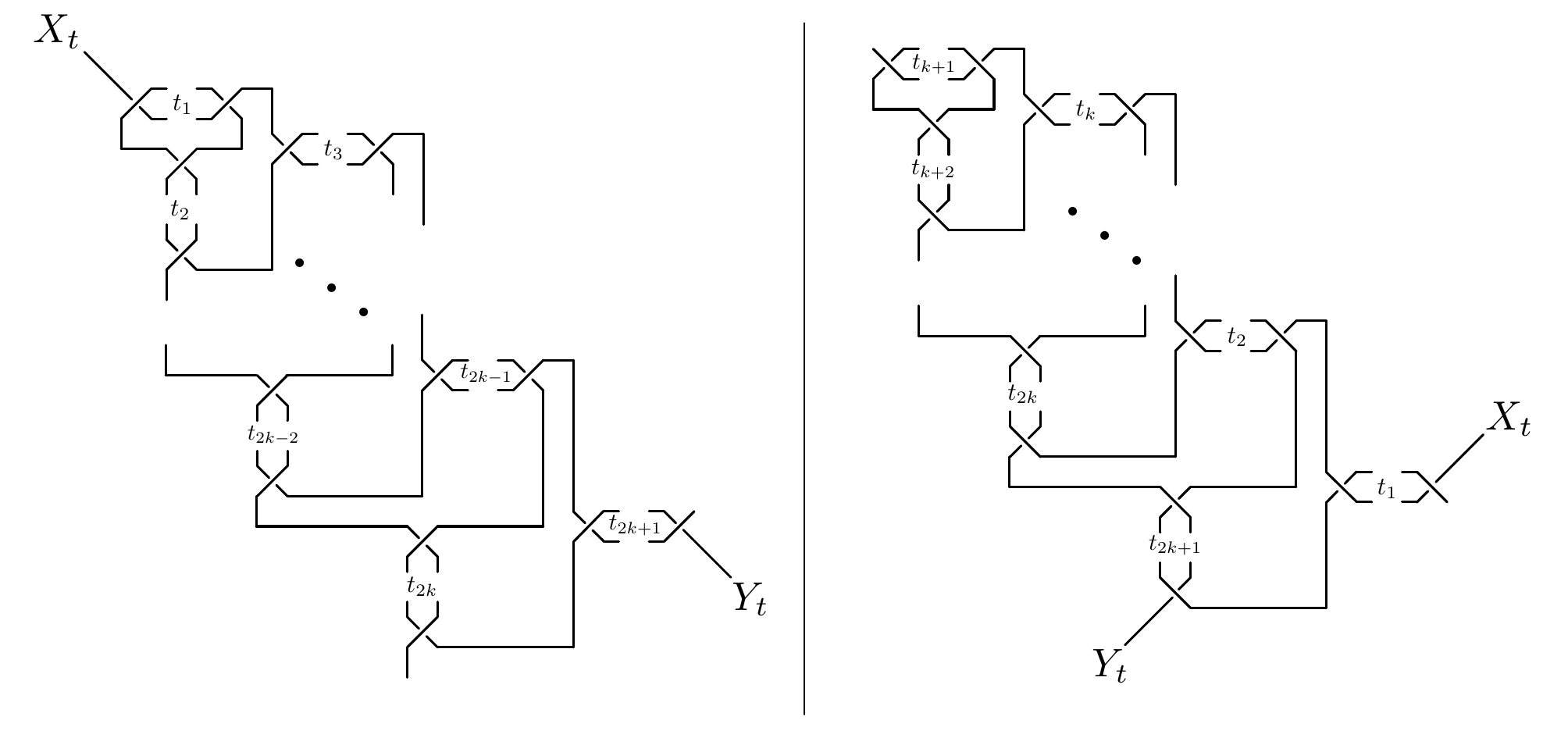}  
\caption{Enumeration of the elementary subtangles of a rational tangle $t$ according to its type.}
\label{fig9}
\end{figure}

\begin{figure}[H]
\centering
\includegraphics[width=0.7\linewidth]{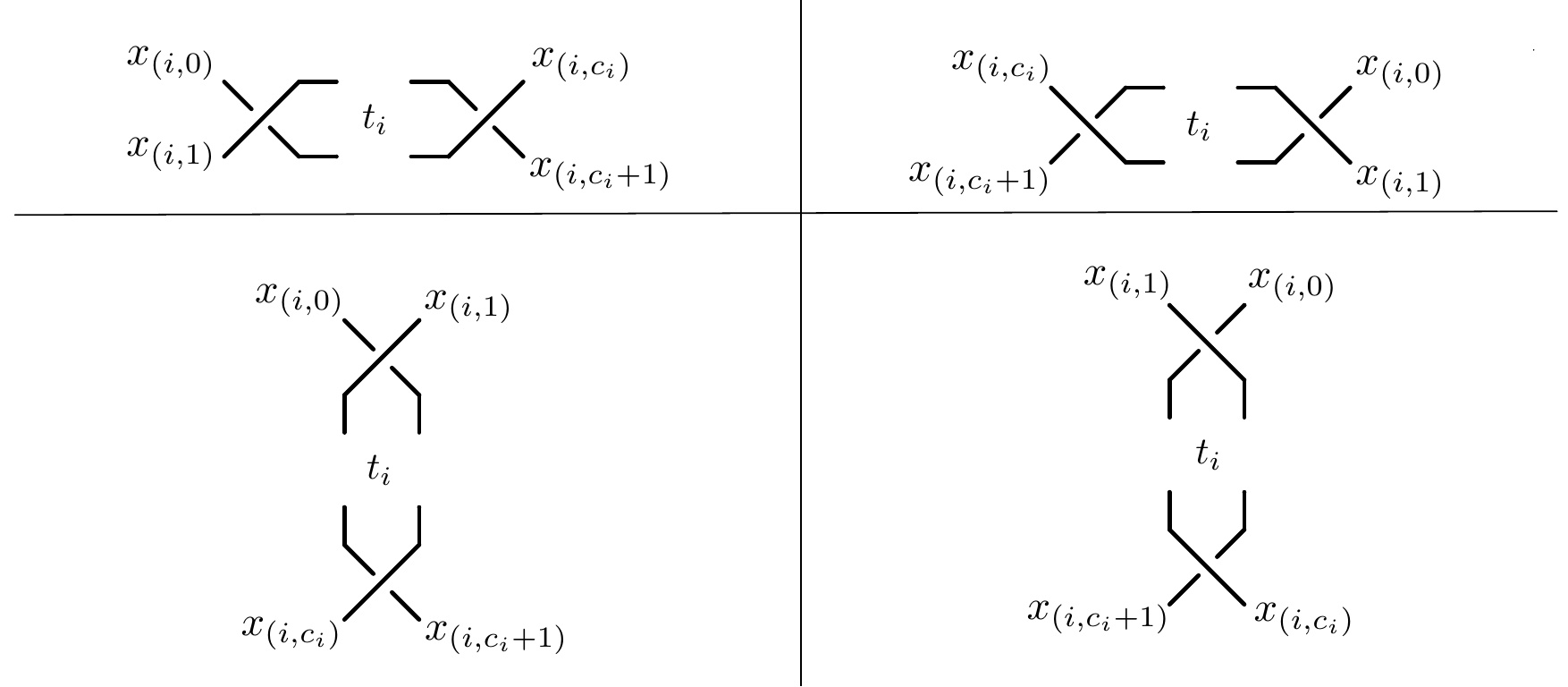}  
\caption{Labelling of the arcs of an elementary subtangle $t_i$ according to its type and fraction.}
\label{fig10}
\end{figure}

Assume that $IQ(L)$ is left-orderable. The following result describes the order induced by any left-order on the elements of $A_{IQ}(t)$.

\begin{prop} Let $D$ be a diagram of a non-split link and $t$ be a tangle part of $D$. If $<$ is any left-order on $IQ(L)$, then there exists a unique element $\diamond$ in $\left\lbrace =,<,> \right\rbrace$ such that 
\begin{align*}
X_t & \diamond x_{(1,1)} \diamond x_{(1,2)} \diamond \hdots \diamond x_{(1,c_1)}\\
& \diamond x_{(2,1)} \diamond x_{(2,2)} \diamond \hdots \diamond x_{(2,c_2)}\\
& \vdots \\
& \diamond x_{(n,1)} \diamond x_{(n,2)} \diamond \hdots \diamond x_{(n,c_n)} \diamond Y_t.
\end{align*}

\end{prop}\label{prop1}
Note that $x>y$ means that $y<x$.\\

To prove this proposition, we will need to show the following lemmas.

\begin{lem}
Let $<$ be any left-order on an involutory quandle $Q$. Suppose that $\diamond$ is one of the symbols in the set $\left\lbrace =,<,> \right\rbrace$. For any elements $x$ and $y$ of $Q$, the following assertions are equivalent.
\begin{enumerate}
\item[(i)] there exists an element $z$ in $Q$ such that $z^x \diamond z^y$.
\item[(ii)] $x \diamond y$.
\item[(iii)] $z^x \diamond z^y$ for every $z \in Q$. 
\item[(iv)] $y^x \diamond x$.
\end{enumerate}
\label{lem1}
\end{lem}

\begin{proof}
At first, note that if $\diamond$ is the symbol $=$, then the assertions are trivially equivalent. So we only consider $\diamond$ to be a fixed element of $\left\lbrace <,> \right\rbrace$.
\item[$(i) \Rightarrow (ii)$] Assume there exists $z$ in $Q$ such that $z^x \diamond z^y$. Suppose on the contrary that $y \diamond x$. Then $z^y \diamond z^x$, which contradicts the assumption. Thus $x \diamond y$.
\item[$(ii) \Rightarrow (iii)$] This exactly expresses the property that $\diamond$ is a left-order on $Q$. 
\item[$(iii) \Rightarrow (i)$] Trivial.
\item[$(iv) \Rightarrow (i)$] Assume that $y^x \diamond x$. We multiply both sides of this inequality on the left by the element $y^x$. We get the following
\begin{align*}
y^x*y^x & \diamond y^x*x \\
\Leftrightarrow y^x & \diamond y \\
\Leftrightarrow y^x & \diamond y^y.
\end{align*}
This makes assertion $(i)$ true.
\item[$(ii) \Rightarrow (iv)$] We will proceed by contraposition. Assume that $x \diamond y^x$. We multiply both sides of this inequality on the left by the element $y^x$. We get 
\begin{align*}
y^x*x & \diamond y^x*y^x \\
\Leftrightarrow y & \diamond y^x \\
\Leftrightarrow y^y & \diamond y^x.
\end{align*}
The last equality and the equivalence $(i) \Leftrightarrow (ii)$ show that $y \diamond x$.
\end{proof}

\begin{lem}Let $<$ be any left-order on an involutory quandle $Q$. Suppose that $\diamond$ is one of the symbols in the set $\left\lbrace =,<,> \right\rbrace$. For any elements $x$ and $y$ of $Q$, the following assertions are equivalent.
\begin{enumerate}
\item[(i)] $x \diamond y$.
\item[(ii)] $y^{(xy)^i} \diamond x^{(xy)^{i+1}}$ for each integer $i \geq 0$.
\end{enumerate}
\label{lem2}
\end{lem}

\begin{proof}
At first, note that if $\diamond$ is the symbol $=$, then the assertions are trivially equivalent. So we only consider $\diamond$ to be a fixed element of $\left\lbrace <,> \right\rbrace$. For $i \ge 0$, let us put $X_i := x^{(xy)^i}$ and $Y_i := y^{(xy)^i}$. By simple computations, we get the following equalities.
\begin{equation}
X_{i+1}=X_i^{Y_i}, \text{ for every integer } i \geq 0.
\label{eq1}
\end{equation}
\begin{equation}
Y_j=Y_{j-1}^{X_j},  \text{ for every integer } j \geq 1.
\label{eq2}
\end{equation}
By the equivalence $(ii) \Leftrightarrow (iv)$ in Lemma \ref{lem1} and (\ref{eq1}), we deduce the following equivalence:
\begin{equation}
Y_i \diamond X_{i+1} \Leftrightarrow X_i \diamond Y_i.
\label{eq3}
\end{equation}
By the equivalence $(ii) \Leftrightarrow (iv)$ in Lemma \ref{lem1} and (\ref{eq2}), we deduce the following equivalence:
\begin{equation}
X_i \diamond Y_i \Leftrightarrow Y_{i-1} \diamond X_i.
\label{eq4}
\end{equation}
By (\ref{eq3}) and (\ref{eq4}), we get the following equivalence:
\begin{equation}
Y_i \diamond X_{i+1} \Leftrightarrow Y_{i-1} \diamond X_i \text{ for every integer } i \geq 1.
\label{eq5}
\end{equation}
The result follows inductively from (\ref{eq5}) and Lemma \ref{lem1}.

\end{proof}

Let $<$ be any left-order on an involutory quandle $Q$. A sequence $\left( x_i \right)_{i\ge 0}$ of elements of $Q$ is said to be \textit{monotonic with respect to} $<$ if for each $i\ge 0$, $x_i\diamond x_{i+1}$, where $\diamond$ is a fixed symbol in the set $\{<,>,=\}$. A sequence of elements of $Q$ is said to be \textit{monotonic} if it is monotonic with respect to any left-order on $Q$. We will say that two monotonic sequences $\left( x_i \right)_{i\ge 0}$ and $\left( y_i \right)_{i\ge 0}$ have the \textit{same monotony} if for any left-order $<$ on $Q$, $x_i\diamond x_{i+1}$ implies $y_i\diamond y_{i+1}$, where $\diamond$ is a fixed symbol in the set $\{<,>,=\}$.

\begin{lem}
Let $Q$ be a left-orderable involutory quandle. Let $\left( x_i \right)_{i \geq 0}$ be a sequence of elements of $Q$ defined recursively by $x_{i+2}=x_{i}*x_{i+1}$. Then $\left( x_i \right)$ is monotonic. 
\label{lem3}
\end{lem}

\begin{proof}
If we put $x_0=x$ and $x_1=y$, then a simple induction shows that 
$x_i = \begin{cases} 
x^{(xy)^{\frac{i}{2}}} &\text{ if } i \text{ is even}, \\
y^{(xy)^{\frac{i-1}{2}}} &\text{ if } i \text{ is odd}.
\end{cases}$

The result then follows by Lemma \ref{lem2}.
\end{proof}

Let $L$ be a non-split link with a diagram $D$ and let $t$ be a tangle part of $D$. Each elementary subtangle $t_i$ of $t$ provides a sequence of generators $A_{IQ}(t_i)=\left\lbrace x_{(i,0)},  x_{(i,1)}, \hdots, x_{(i,c_i + 1)}  \right\rbrace$ such that $x_{(i,j-1)}*x_{(i,j)}=x_{(i,j+1)}$ for every $1 \leq j \leq c_i$, where $c_i$ is the number of crossings in the elementary tangle $t_i$ (see Figure \ref{fig10}). By Lemma \ref{lem3}, the sequence $\left( x_{(i,\bullet)} \right)$ is monotonic for all $1 \leq i \leq n$. 

\begin{lem}
Assume that $t$ is negative with odd length $2k+1$. Let $k_i = 2(k+1)-i$, $1 \leq i \leq k$. Assume that $IQ(L)$ is left orderable. If there exists $ 1 \leq i \leq k$ such that the sequences $\left( x_{(i+1,\bullet)} \right)$ and $\left( x_{(k_i,\bullet)} \right)$ have the same monotony, then the sequences $\left( x_{(i,\bullet)} \right)$ and $\left( x_{(k_{i-1},\bullet)} \right)$ also have the same monotony as $\left( x_{(k_i,\bullet)} \right)$. 
\label{lem4}
\end{lem}

\begin{proof}
Let $<$ be any left-order on the elements of $IQ(L)$ and assume that there exists $1 \leq i \leq k$ such that the sequence $\left( x_{(i+1,\bullet)} \right)$ have the same monotony as $\left( x_{(k_i,\bullet)} \right)$ with respect to $<$. 

Let $\diamond$ be a fixed symbol in the set $\left\lbrace =,<,> \right\rbrace$ such that $x_{(i+1,0)} \diamond x_{(i+1,1)}$. Since $x_{(i+1,0)}=x_{(i,c_i)}$, $x_{(i+1,1)}=x_{(k_i,0)}$, and $x_{(k_i,c_{k_i})}=x_{(i,c_i+1)}$ (see Figure \ref{figadd1}), then our assumptions yield the following.
\begin{align}
x_{(i,c_i)} & \diamond x_{(k_i,0)} \nonumber \\
\Rightarrow x_{(i,c_i)} & \diamond x_{(k_i,0)} \diamond x_{(k_i,c_{k_i})} \nonumber \\
 \Leftrightarrow x_{(i,c_i)} & \diamond x_{(k_i,0)} \diamond x_{(i,c_i+1)}.
\label{eq6}
\end{align}
The result (\ref{eq6}) shows that the sequence $\left( x_{(i,\bullet)} \right)$ has the same monotony as  $\left( x_{(k_i,\bullet)} \right)$ and $\left( x_{(i+1,\bullet)} \right)$. Since $x_{(i,1)}=x_{(k_{i-1},0)}$, $x_{(i,c_i+1)}=x_{(k_i,c_{k_i})}$, and $x_{(k_i,c_{k_i}+1)}=x_{(k_{i-1},1)}$ (see Figure \ref{figadd1}), then (\ref{eq6}) implies that 

\begin{align}
x_{(i,1)} & \diamond x_{(i,c_i)} \diamond x_{(k_i,0)} \diamond x_{(i,c_i+1)} \nonumber \\
\Leftrightarrow x_{(k_{i-1},0)} & \diamond x_{(i,c_i)} \diamond x_{(k_i,0)} \diamond x_{(k_i,c_{k_i})} \nonumber \\
\Rightarrow x_{(k_{i-1},0)} & \diamond x_{(k_i,c_{k_i})} \diamond x_{(k_i,c_{k_i}+1)} \nonumber \\
\Leftrightarrow x_{(k_{i-1},0)} & \diamond x_{(k_i,c_{k_i})} \diamond x_{(k_{i-1},1)}.
\label{eq7}
\end{align}
The result (\ref{eq7}) shows that the sequence $\left( x_{(k_{i-1},\bullet)} \right)$ has the same monotony as $\left( x_{(i,\bullet)} \right)$. This completes the proof.
\end{proof}

\begin{figure}[H]
\centering
\includegraphics[width=0.65\linewidth]{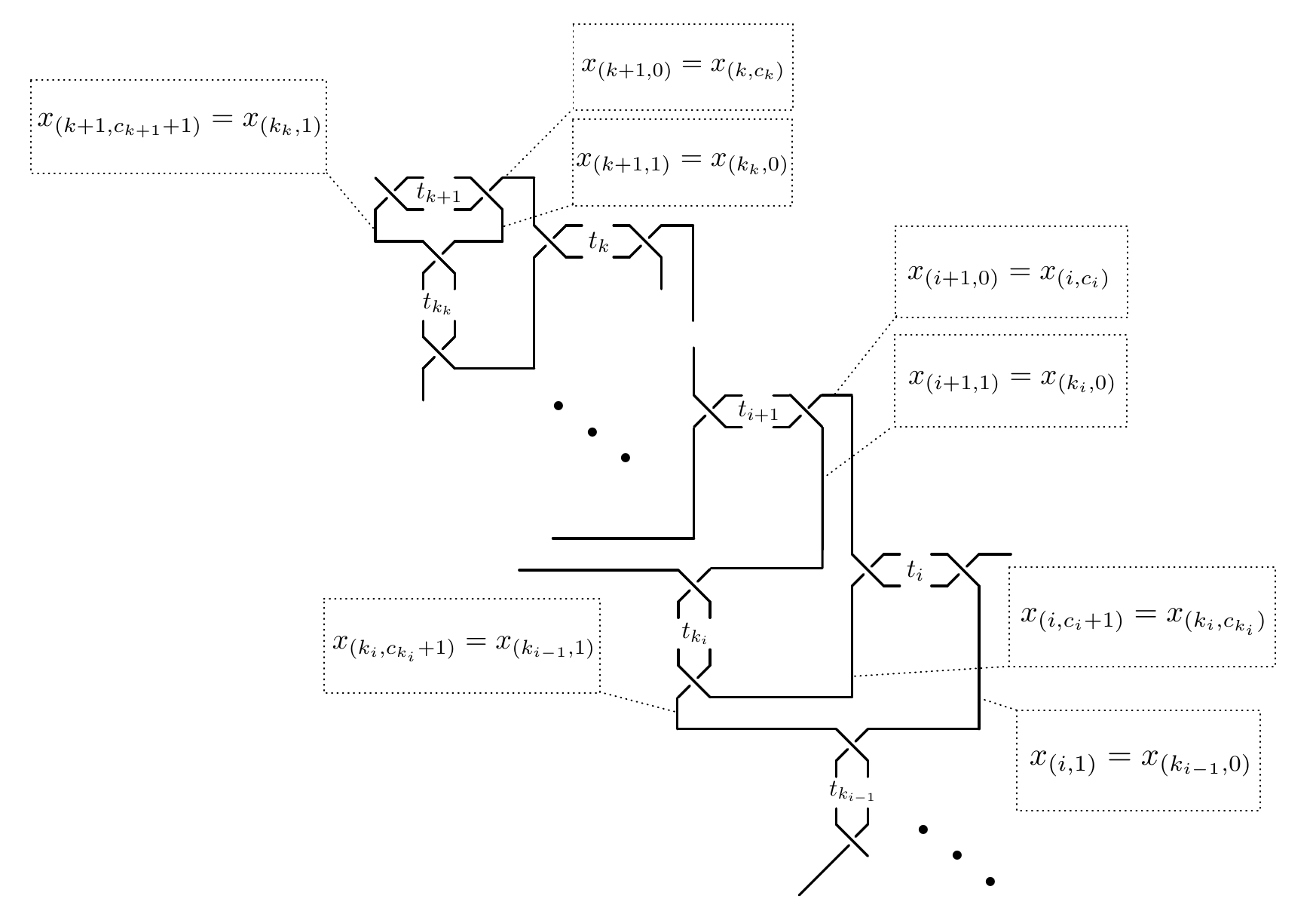}
\caption{Identities used in the proofs of Lemma \ref{lem4} and Proposition \ref{prop1}.}
\label{figadd1}
\end{figure}

\begin{proof}[\textbf{Proof of Proposition \ref{prop1}}]
Let $L$ be a non-split link with a diagram $D$ and assume that $IQ(L)$ is left-orderable. Let $t$ be a tangle part of $D$. We assume without loss of generality that the length $n$ of $t$ is odd. 

\item[\underline{Case 1}: $t$ is positive.] 

Label the arcs of $t$ as in the left of Figure \ref{fig9}. The result we want to prove is equivalent to the assertion that the sequences $\left( x_{(i,\bullet)} \right)$ and $\left( x_{(i+1,\bullet)} \right)$ are of the same monotony for each $1 \leq i \leq n-1$. We proceed by induction on $i$.

Let $<$ be any left-order on $IQ(L)$ and $\diamond$ be a fixed symbol in the set $\lbrace =,<,> \rbrace$ such that $x_{(1,1)} \diamond x_{(1,c_1 + 1)}$. Since $x_{(1,1)} = x_{(2,0)}$ and $x_{(1,c_1+1)} = x_{(2,1)}$ (see Figure \ref{figadd2}), then $x_{(2,0)} \diamond x_{(2,1)}$. Since $<$ is an arbitrary left-order then the sequences $\left( x_{(1,\bullet)} \right)$ and $\left( x_{(2,\bullet)} \right)$ are of the same monotony. This completes the proof of the induction basis.

Suppose that there exists $1 \leq i \leq n-2$ such that $\left( x_{(i,\bullet)} \right)$ and $\left( x_{(i+1,\bullet)} \right)$ are of the same monotony. Let $<$ be any left-order on $IQ(L)$ and $\diamond$ be a fixed symbol in the set $\lbrace =,<,> \rbrace$ such that 

\begin{equation}
x_{(i,c_i)} \diamond x_{(i,c_i+1)}.
\label{eq8}
\end{equation}

The induction hypothesis implies the following.

\begin{equation}
x_{(i+1,1)} \diamond x_{(i+1,c_{i+1}+1)}.
\label{eq9}
\end{equation}

Since $x_{(i,c_i)} = x_{(i+2,0)}$, $ x_{(i,c_i+1)}=x_{(i+1,1)}$, and $x_{(i+1,c_{i+1}+1)}=x_{(i+2,1)}$ (see Figure \ref{figadd2}), then (\ref{eq8}) and (\ref{eq9}) together imply that $x_{(i+2,0)} \diamond x_{(i+2,1)}$. Since $<$ is an arbitrary left-order, then  $\left( x_{(i+1,\bullet)} \right)$ and $\left( x_{(i+2,\bullet)} \right)$ are of the same monotony. 

This shows that the sequences $\left( x_{(i,\bullet)} \right)$ and $\left( x_{(i+1,\bullet)} \right)$ are of the same monotony for each $1 \leq i \leq n-1$. This completes the proof of this case. 

\begin{figure}[H]
\centering
\includegraphics[width=0.65\linewidth]{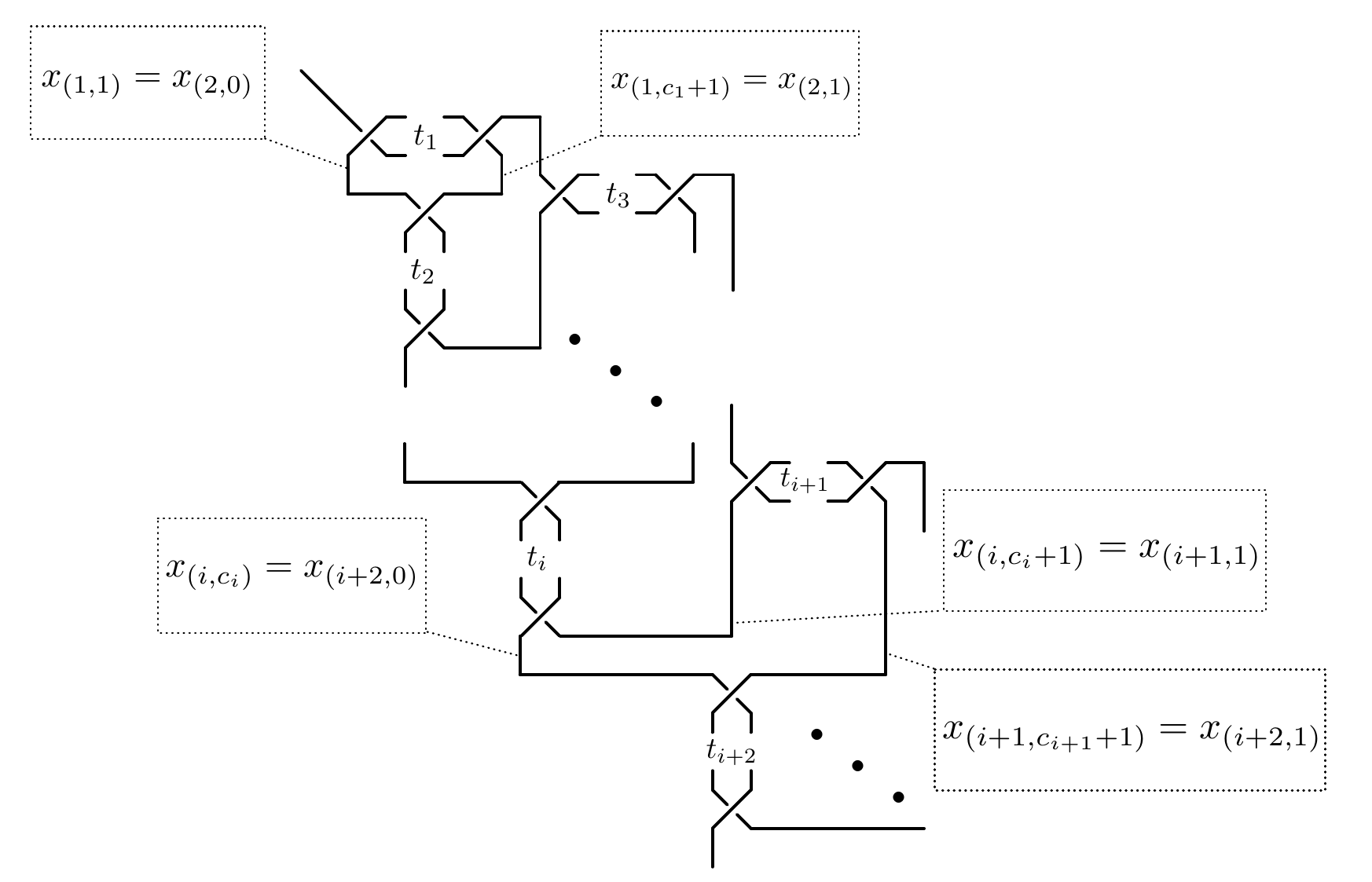}
\caption{Used Identities.}
\label{figadd2}
\end{figure}

\item[\underline{Case 2}: $t$ is negative.] 

Label the arcs of $t$ as in the right of Figure \ref{fig9}. Put $n=2k+1$ and $k_i=2(k+1)-i$ for each $1 \leq i \leq k$. We will show that the sequences $\left( x_{(k+1,\bullet)} \right)$ and $\left( x_{(k_k,\bullet)} \right)$ have the same monotony. 

Let $<$ be any left-order on $IQ(L)$ and $\diamond$ be a fixed symbol in the set $\lbrace =,<,> \rbrace$ such that $x_{(k+1,0)} \diamond x_{(k+1,1)}$. Since the sequence $\left( x_{(k+1,\bullet)} \right)$ is monotonic, then the assumption yields the following 
\begin{equation}
x_{(k+1,0)} \diamond x_{(k+1,1)} \diamond x_{(k+1,c_{k+1}+1)}.
\label{eq10}
\end{equation}

Since $x_{(k+1,0)}=x_{(k,c_k)}$, $x_{(k+1,1)}=x_{(k_k,0)}$, and $x_{(k+1,c_{k+1}+1)}=x_{(k_k,1)}$ (see Figure \ref{figadd1}), then (\ref{eq10}) is equivalent to the following

\begin{equation}
x_{(k,c_k)} \diamond x_{(k_k,0)} \diamond x_{(k_k,1)}.
\label{eq11}
\end{equation}

Since $<$ is an arbitrary left-order, then the result (\ref{eq11}) implies that $\left( x_{(k+1,\bullet)} \right)$ and $\left( x_{(k_k,\bullet)} \right)$ have the same monotony. The result follows inductively by Lemma \ref{lem4}.

\end{proof}

Now we are able to define the coarse presentation of the involutory quandle of a non-split link.
\begin{mydef}
Let $D$ be a link diagram representing a non-split link $L$. Let $\langle x:x\in A(D) | r_1, \hdots, r_k \rangle$ be the presentation of $IQ(L)$ given by the diagram $D$ which has $k$ crossings. The \textit{coarse presentation} of $IQ(L)$ associated to the diagram $D$, denoted by $\mathcal{IQ}(L,D)$, is the set of generators and relations given as follows.

\item[\textbf{Generators:}] $x\in A(D)$ (the set of arcs of $D$).
\item[\textbf{Crossing relations:}] $r_1, \hdots, r_k$ (the set of crossing relations of $D$).
\item[\textbf{Coarse relations:}] $A_{IQ}(t) \subset \left\lbrace | X_t,Y_t | \right\rbrace$ for each tangle part $t$ of $D$.
\label{mydef1}
\end{mydef}

Note that the coarse relations are well defined by Proposition \ref{prop1}. A coarse relation $A_{IQ}(t) \subset \left\lbrace | X_t,Y_t | \right\rbrace$ is said to be \textit{trivial} if $X_t = Y_t$.
\begin{empl}
 Below is an example of a coarse presentation.
\begin{figure}[H]
\centering
\includegraphics[width=0.75\linewidth]{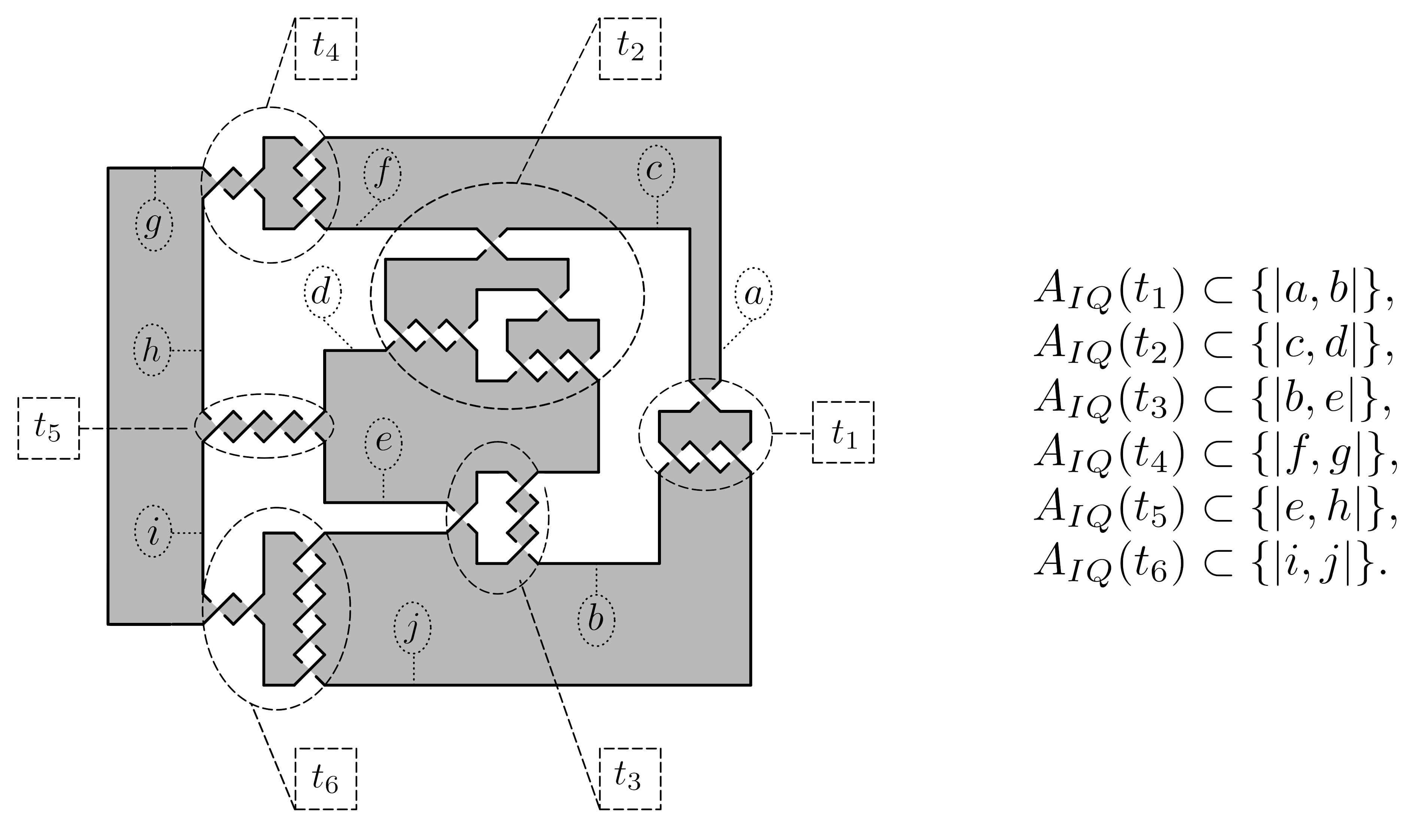}
\caption{A tangle-strand decomposition of a link diagram together with some labeled arcs (left) and the corresponding coarse relations (right).}
\end{figure}
\end{empl}
\subsection{The criterion}
The following result gives a necessary condition for the existence of a left-order on the involutory quandle of a non-split link. 
\begin{theo}
Let $D$ be a diagram of a non-split link $L$. Let $\mathcal{IQ}(L,D)$ be the coarse presentation of the involutory quandle $IQ(L)$ associated to $D$. If $IQ(L)$ is left-orderable, then there exists a non-trivial coarse relation in $\mathcal{IQ}(L,D)$. 
\label{theo_1}
\end{theo}

\begin{proof}
Let $L$ be a non-split link with a diagram $D$. Assume that $IQ(L)$ is left-orderable and suppose on the contrary that all the coarse relations of $\mathcal{IQ}(L,D)$ are trivial, meaning that for every tangle part $t$ of $D$, the set $A_{IQ}(t)$ contains a single element which we denote by $x_t$. Let $R$ be a region of $D$ such that $\partial R$ consists of $n$ tangle parts. Fix a tangle $t_1 \in \partial R$. Travelling along $\partial R$ in the clockwise direction starting from $t_1$ provides an enumeration $\partial R = \left\lbrace t_1, \hdots ,t_n \right\rbrace$ such that $A(t_i) \cap A(t_{i+1}) \neq \emptyset$ for every $1 \leq i \leq n-1$. Since $A_{IQ}(t)=\left\lbrace x_t \right\rbrace$ for every tangle part $t$, then $A_{IQ}(t_i) = A_{IQ}(t_{i+1})$ and $x_{t_i} = x_{t_{i+1}}$ for every $1 \leq i \leq n-1$. Thus $A_{IQ}(R)$ contains a single element which we denote by $x_R$. If $R$ is the only region of $D$ then $A_{IQ}(D) = A_{IQ}(R) = \left\lbrace x_R \right\rbrace$. Which means that $IQ(L)$ is trivial with a single generator. This contradicts the fact that $IQ(L)$ is left-orderable. We thereby assume that $D$ has other regions different from $R$. Let $R^{'}$ be a region of $D$ other than $R$. Identify the regions of $D$ with the vertices of the dual graph $\left( \Gamma_D \right)^*$. By connectedness of the graph $\Gamma_D$ and Theorem 6 in \cite{whitney2009non}, there exists a chain of edges connecting the vertices of $\left( \Gamma_D \right)^*$ corresponding to $R$ and $R^{'}$. Equivalently, there exist finitely many regions $R=R_1, \hdots , R_m=R^{'}$, such that $\partial R_j \cap \partial R_{j+1} \neq \emptyset$ for every $1 \leq j \leq m-1$. More precisely, the set $\partial R_j \cap \partial R_{j+1}$ contains the tangle part corresponding to the edge of $\left( \Gamma_D \right)^*$ connecting the vertices that correspond to $R_j$ and $R_{j+1}$. Let $t^{'}_j$ denote the described tangle part. It is clear that $x_{R_j} = x_{R_{j+1}}= x_{t^{'}_{j}}$ for every $1 \leq j \leq m-1$. So finally $ x_R = x_{R^{'}}$. And since the regions $R$ and $R^{'}$ were chosen arbitrarily, then $A_{IQ}(D)$ contains a single element and thus $IQ(L)$ is trivial with a single generator. This contradicts the fact that $IQ(L)$ is left-orderable which implies that our supposition is false. Thus, there exists at least one non-trivial coarse relation in $\mathcal{IQ}(L,D)$.
\end{proof}

\section{Applications}
\subsection{Alternating links.}
Raudal et al. (Theorem 8.1 in \cite{raundal}) showed that the involutory quandle of an alternating link $L$ is not left-orderable if there exists a reduced alternating diagram $D$ of $L$ where the sets $A(D)$ and $A_{IQ}(D)$ are in one-to-one correspondence. Theorem \ref{theo_1} allows us to significantly improve this result by showing that the condition of the existence of the one-to-one correspondence between $A(D)$ and $A_{IQ}(D)$ is finally unnecessary. Then, we have the following theorem.

\begin{theo}
If $L$ is a non-trivial alternating link, then $IQ(L)$ is not left-orderable.
\label{theo_2}
\end{theo}

\begin{proof}
Let $L$ be a non-split alternating link with an alternating reduced diagram $D$. Suppose on the contrary that $IQ(L)$ is left-orderable and let $<$ be any left-order on $IQ(L)$. Denote by $x_0$ the minimum of the generators $A_{IQ}(D)$ of $IQ(L)$ with respect to $<$ and let $t_0$ be a tangle part of $D$ such that $x_0 \in A_{IQ}(t_0)$.

Let $R$ be a fixed region of $D$ such that $t_0 \in \partial R$. We denote $t_0,\dots, t_{n+1}=t_0$ the tangles that we go through successively when we move along $\partial R$ starting from $t_0$ in the clockwise direction. Note that $A(t_i) \cap A(t_{i+1}) \neq \emptyset$ for each $0 \leq i \leq n$. By Proposition \ref{prop1}, we know that the minimum $x_0$ belongs to $\{X_{t_0},Y_{t_0}\}$. Without loss of generality we can suppose that $x_0=X_{t_0}$ and $X_{t_i} \in A(t_i) \cap A(t_{i+1})$ for every $0 \leq i \leq n$ (see Figure \ref{fig14}). By the coarse relations in $\mathcal{IQ}(L,D)$, we get that

\begin{equation}
X_{t_i} \in A_{IQ}(t_{i+1}) \subset \left\lbrace | X_{t_{i+1}},Y_{t_{i+1}} | \right\rbrace \text{ for every } 0 \leq i \leq n-1. \label{eq12} 
\end{equation}

\begin{figure}[H]
\centering
\includegraphics[width=0.25\linewidth]{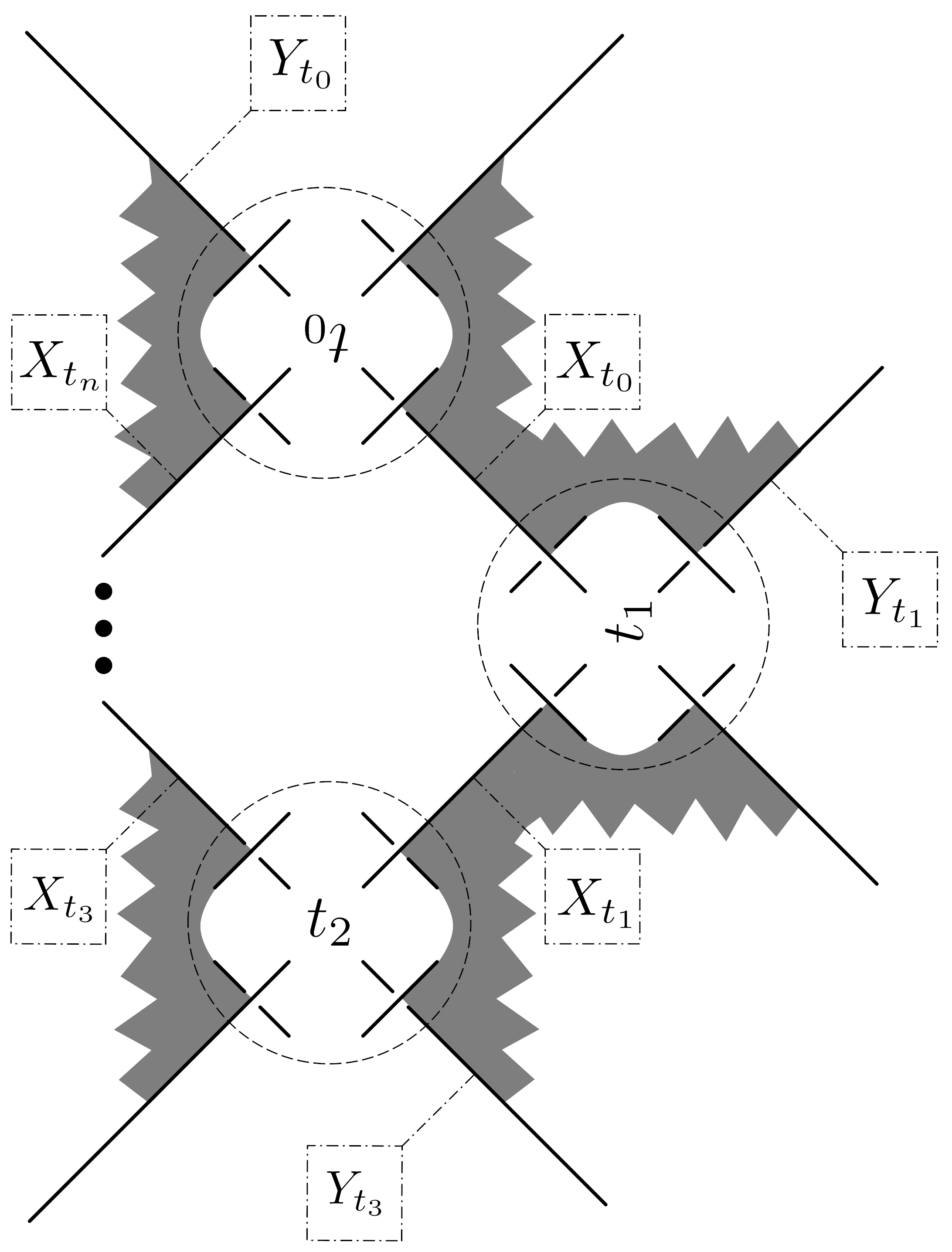}
\caption{The boundary of a region of an alternating diagram.}
\label{fig14}
\end{figure}

By (\ref{eq12}), we have that $x_0 = X_{t_0} \in \left\lbrace | X_{t_1},Y_{t_1} | \right\rbrace$. Since $x_0$ is the minimum element of $A_{IQ}(D)$ with respect to $<$ and $X_{t_0}$ is distinct from the arcs $X_{t_1}$ and $Y_{t_1}$ in $A(t_1)$, then by Proposition \ref{prop1} we have $x_0 = X_{t_1}=Y_{t_1}$. This implies that the coarse relation $A_{IQ}(t_1) \subset \left\lbrace | X_{t_1},Y_{t_1} | \right\rbrace$ is trivial i.e. $A_{IQ}(t_1)=\left\lbrace x_0 \right\rbrace$. By (\ref{eq12}), we get $x_0=X_{t_1} \in \left\lbrace | X_{t_2},Y_{t_2} | \right\rbrace$. By using the same argument we get that $x_0=X_{t_2}=Y_{t_2}$, which implies that the coarse relation $A_{IQ}(t_2) \subset \left\lbrace | X_{t_2},Y_{t_2} | \right\rbrace$ is also trivial. By iterating this argument, we show that all the coarse relations $A_{IQ}(t_i) \subset \left\lbrace | X_{t_i},Y_{t_i} | \right\rbrace$ are trivial and then $A_{IQ}(R)=\left\lbrace x_0 \right\rbrace$.

Let $R'$ be another region of $D$. There are two posibilities:
\item[\underline{Case 1:} $\partial R \cap \partial R^{'} \neq \emptyset$.] Let $t^{'}$ be an element of $\partial R \cap \partial R^{'}$. Since $A_{IQ}(t^{'}) \subset A_{IQ}(R)=\left\lbrace x_0 \right\rbrace$, then $R^{'}$ satisfies the same assumptions as $R$. So, all coarse relations corresponding to the tangle parts constituting $\partial R^{'}$ are trivial.
\item[\underline{Case 2:} $\partial R \cap \partial R^{'} =\emptyset$.]
  We know that the regions of $D$ are identified with the vertices of the connected dual graph $\left( \Gamma_D \right)^*$. By connectedness of $\left( \Gamma_D \right)^*$ and Theorem 6 in \cite{whitney2009non}, we can find a chain of edges connecting the vertices of $\left( \Gamma_D \right)^*$ corresponding to $R$ and $R^{'}$.\\
  Equivalently, there exist finitely many regions $R=R_0,R_1, \hdots , R_m=R^{'}$, such that for every $0 \leq j \leq m$, the set $\partial R_j \cap \partial R_{j+1}$ contains the tangle part corresponding to the edge of $\left( \Gamma_D \right)^*$ connecting the vertices that correspond to $R_j$ and $R_{j+1}$. As in Case 1, since $\partial R \cap \partial R_{1} \neq \emptyset$, then all coarse relations corresponding to the elements of $\partial R_1$ are trivial. By iterating the same argument we conclude that the coarse relations corresponding to the elements of $\partial R_j$ are trivial for each $j$, $0 \leq j \leq m$. In particular, the coarse relations corresponding to the elements of $\partial R^{'}$ are trivial. 

Since the region $R^{'}$ was chosen arbitrarily, then we conclude that all the coarse relations of $\mathcal{IQ}(L,D)$ are trivial, which is a contradiction by Theorem \ref{theo_1}. 

If $L^{'}$ is a non-trivial split alternating link, then the involutory quandle of any non-split alternating sub-link $L$ of $L^{'}$ in not left-orderable. Thus, $IQ(L^{'})$ has a sub-quandle that is not left-orderable. Consequently, $IQ(L^{'})$ is not left-orderable.
\end{proof}

\subsection{Augmented alternating links.}
Let $D$ be a reduced alternating diagram of a non-trivial non-split alternating link $L$. Let $J$ be an embedded circle in the complement of $L$ in $S^3$ such that $J$ intersects the projection plane in two points and bounds a disk containing two points of $D$ in its interior and that lies in the perpendicular plane to the projection plane. The link $L \cup J$ is called an \textit{augmentation} of $L$. If $L$ is prime and non-isotopic to a torus link $T(2,k)$ for any integer $k$, then the link $L \cup J$ is called an \textit{augmented alternating link}. Recall that the torus link $T(2,k)$ is the numerator closure of the elementary tangle $-k$. Augmented alternating links were introduced by Adams in \cite{adams} where he shows that augmented alternating links are non-alternating and non-split.

It has been shown in \cite{hoste2017involutory} that the involutory quandle of some augmented alternating Montesinos links are finite. On the other hand, certain augmentations of a rational link also have finite involutory quandles as shown in \cite{mellor2022finite}. This provides infinitely many examples of augmented alternating links having non-left-orderable involutory quandles. In fact, we prove that the involutory quandle of any augmented alternating link is not left-orderable as stated in the following theorem.

\begin{theo}
If $L$ is an augmented alternating link, then $IQ(L)$ is not left-orderable. 
\label{theo_3}
\end{theo}
\begin{proof}
Let $L$ be an augmented alternating link. It is easy to see that there exists a connected alternating tangle $T$ such that $L$ is isotopic to $N\left( -\frac{1}{2} + \frac{1}{2} + \frac{1}{T_c} \right)$. Without loss of generality, the tangle $T$ can be assumed of type 2. Denote by $y$ and $z$ the arcs of $T$ passing respectively through the points $NW$ and $SE$ and label by $\alpha$ and $\beta$ the two arcs of the augmentation circle as in Figure \ref{fig15}. 

\begin{figure}[H]
\centering
\includegraphics[width=0.5\linewidth]{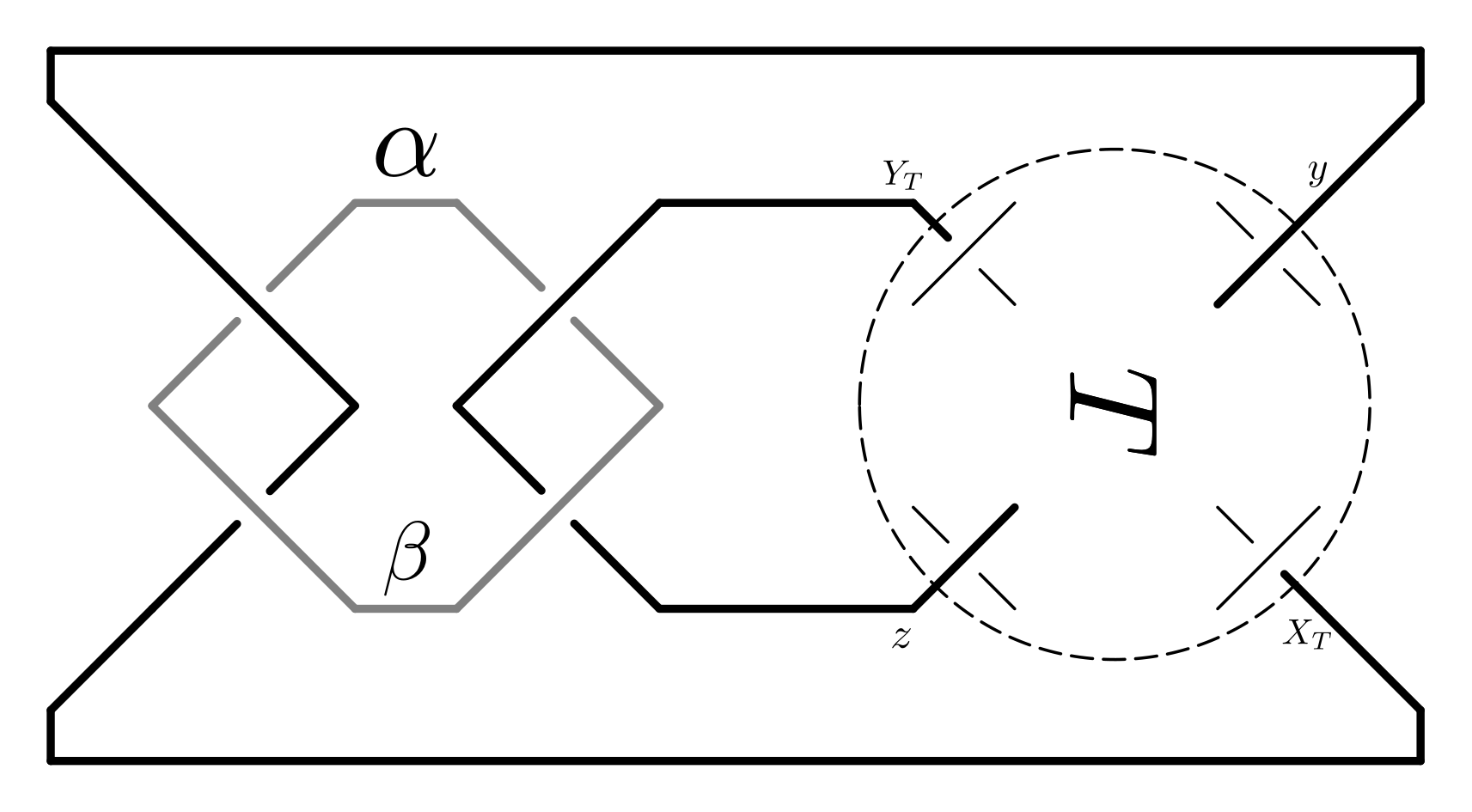}
\caption{The diagram $n\left( -\frac{1}{2} + \frac{1}{2} + \frac{1}{T_c} \right)$ and the arcs $\alpha$ and $\beta$. }
\label{fig15}
\end{figure} 

Suppose on the contrary that $IQ(L)$ is left-orderable. The generating set of $IQ(L)$ is $A(T) \cup \left\lbrace \alpha, \beta \right\rbrace$ and in addition to the relations $r_c, c\in C(T)$, we have the relations $\beta=\alpha^{Y_T}=\alpha^{y}, Y_T^{\beta} = z , y^{\beta}= X_{T}$. By Lemma \ref{lem1}, these relations imply that $Y_T=y$ and $X_T=z$. On the other hand, the involutory quandle whose generating set is $A(T)$ subject to the relations $Y_T=y,X_T=z$, and $r_c, c\in C(T)$ is exactly the involutory quandle of the alternating link $D(T)$. This enable us to use an analogous argument as in the proof of Theorem \ref{theo_2} to show that the coarse relations in $\mathcal{IQ}(L,D)$ corresponding to the rational subtangles of the tangle $T$ are all trivial. Hence, the subset $A_{IQ}(T)$ of generators of $IQ(L)$ contains a single element. This clearly implies that all the coarse relations are trivial which is a contradiction by Theorem \ref{theo_1}. Thereby, the quandle $IQ(L)$ is not left-orderable.
\end{proof}

\subsection{Quasi-alternating $3$-braid closures.}
Let $B_3$ be the $3$-braid group. It is known that $B_3$ is generated by the elementary braids $\sigma_1$ and $\sigma_2$ depicted in Figure \ref{fig11} subject to the relation $\sigma_1 \sigma_2 \sigma_1 = \sigma_2 \sigma_1 \sigma_2$.

\begin{figure}[H]
\centering
\includegraphics[width=0.6\linewidth]{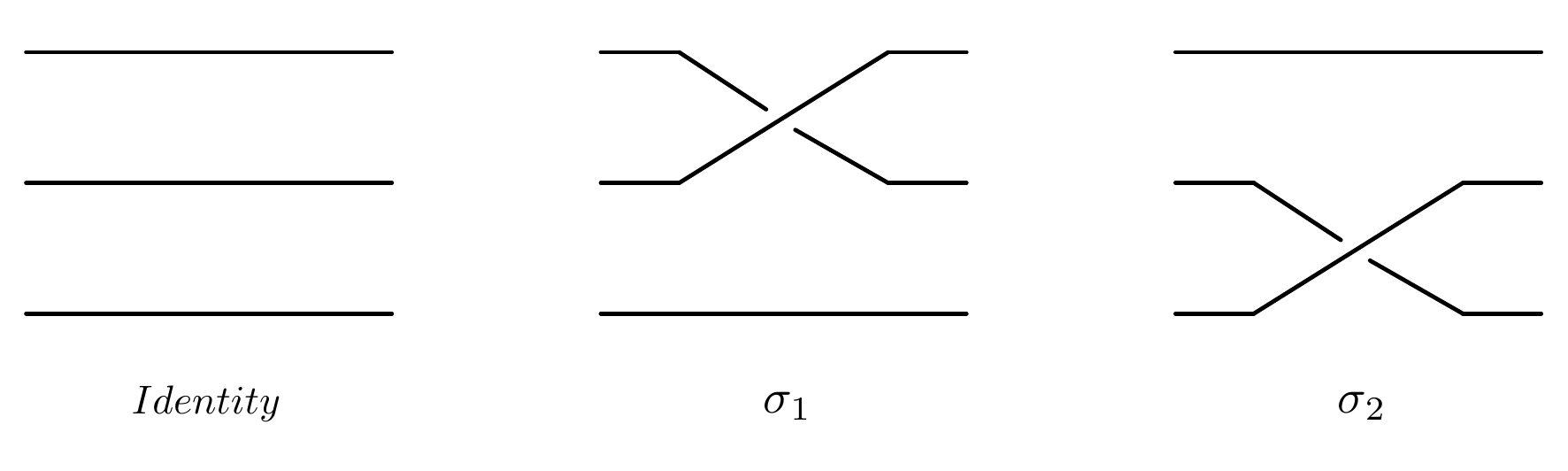}
\caption{The identity element and the generators of the group $B_3$.}
\label{fig11}
\end{figure}

Let $h$ denote the braid $(\sigma_1 \sigma_2)$. We denote by $L\left((t_1,s_1),\hdots,(t_n,s_n),h^l\right)$ the link represented by the diagram depicted in Figure \ref{fig12} where $t_1, \hdots , t_n$ and $s_1, \hdots , s_n$ are rational tangles of type 2. It is known that a split link has a zero determinant. We check that the determinant of $L\left((t_1,s_1),\hdots,(t_n,s_n),h^l\right)$ is non-zero. So it is non-split.

\begin{figure}[H]
\centering
\includegraphics[width=0.5\linewidth]{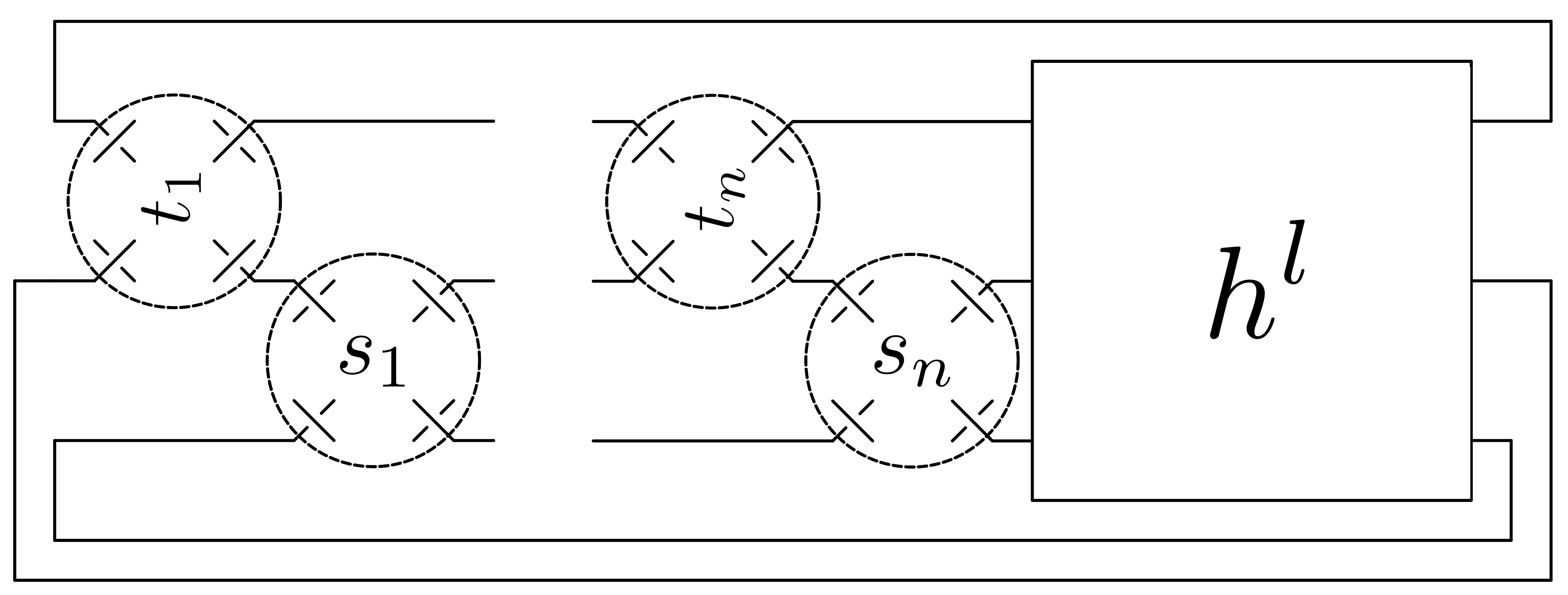}  
\caption{The link $L\left((t_1,s_1),\hdots,(t_n,s_n),h^l\right)$.}
\label{fig12}
\end{figure}

\begin{theo}
Let $L$ be the link $L\left((t_1,s_1),\hdots,(t_n,s_n),h^{l})\right)$ depicted in Figure \ref{fig12}. 
\begin{enumerate}
\item[(i)] If $l=3$ and $| s_n | \geq 1$, then $IQ(L)$ is not left-orderable.  
\item[(ii)] If $l=-3$ and $| t_1 | \leq 1$, then $IQ(L)$ is not left-orderbale.
\end{enumerate}

\label{theo_4}
\end{theo}
To prove Theorem \ref{theo_4}, we need the following lemma which is easily deduced from the transitivity of the order $<$.

\begin{lem}[Self duplicating templates]
Let $L$ be the link $L \left( (t_1,s_1), \hdots,(t_n,s_n), h^{3} \right) $ presented by the diagram $D$ depicted in Figure \ref{fig12}. Assume that $IQ(L)$ is left-orderable and let $<$ be any left-order on $IQ(L)$. We fix the three arcs $x$, $y$, and $z$ of the full-twist $h^3$ as in Figure \ref{fig16}. We choose the isotopy classes of the tangles $t_i$ and $s_i$ such that $X_{t_1}=x$, $Y_{s_1}=z$, $X_{s_n}=y^{zyx}$, $A(t_i) \cap A(t_{i-1})=\left\lbrace X_{t_i} \right\rbrace$ for $2 \leq i \leq n$, $A(t_i) \cap A(s_i)=\left\lbrace Y_{t_i} \right\rbrace$ for $2 \leq i \leq n$, $A(s_i) \cap A(t_{i+1})= \left\lbrace X_{s_i} \right\rbrace$ for $1 \leq i \leq n-1$, and $A(s_i) \cap A(s_{i-1})=\left\lbrace Y_{s_i} \right\rbrace$ for $2 \leq i \leq n$. With these notations, we have the following:

\begin{itemize}
\item[(S.D.T.1)] If there exists an integer $1 \leq i \leq n$ such that $A_{IQ}(t_i) \subset \left( \infty,Y_{t_i} \right]_<$ and $A_{IQ}(s_i) \subset \left( \infty,X_{s_i} \right]_<$, then $A_{IQ}(t_j) \subset \left( \infty,Y_{t_j} \right]_<$ and $A_{IQ}(s_j) \subset \left( \infty,X_{s_j} \right]_<$ for every $i \leq j \leq n$. 
\item[(S.D.T.2)] If there exists an integer $1 \leq i \leq n$ such that $A_{IQ}(t_i) \subset \left( \infty,X_{t_i} \right]_<$ and $A_{IQ}(s_i) \subset \left( \infty,Y_{s_i} \right]_<$, then $A_{IQ}(t_j) \subset \left( \infty,X_{t_j} \right]_<$ and $A_{IQ}(s_j) \subset \left( \infty,Y_{s_j} \right]_<$ for every $i \leq j \leq n$. 
\item[(S.D.T.3)] If there exists an integer $1 \leq i \leq n$ such that $A_{IQ}(t_i) \subset \left( \infty,X_{t_i} \right]_<$ and $A_{IQ}(s_i) \subset \left( \infty,X_{s_i} \right]_<$, then $A_{IQ}(t_j) \subset \left( \infty,X_{t_j} \right]_<$ and $A_{IQ}(s_j) \subset \left( \infty,X_{s_j} \right]_<$ for every $1 \leq j \leq i$. 
\item[(S.D.T.4)] If there exists an integer $1 \leq i \leq n$ such that $A_{IQ}(t_i) \subset \left( \infty,Y_{t_i} \right]_<$ and $A_{IQ}(s_i) \subset \left( \infty,Y_{s_i} \right]_<$, then $A_{IQ}(t_j) \subset \left( \infty,Y_{t_j} \right]_<$ and $A_{IQ}(s_j) \subset \left( \infty,Y_{s_j} \right]_<$ for every $1 \leq j \leq i$. 

\end{itemize}

\label{lem5}
\end{lem}

Let $D$ be the link diagram depicted in Figure \ref{fig12}. By \textit{self duplicating templates} we mean how the orientations of two fixed adjacent edges of $\Gamma_{D,<}$ corresponding to $t_i$ and $s_i$ impose the orientations of either the edges at the right or at the left of these fixed edges by duplication. This observation will be a key technique in the proof of Theorem \ref{theo_4}. 

\begin{proof}[Proof of Theorem \ref{theo_4}]
Let $L$ denote the link $L\left((t_1,s_1),\hdots,(t_n,s_n),h^{l})\right)$ which admits a diagram $D$ as depicted in Figure \ref{fig12}. We fix three arcs $x$, $y$, and $z$ of $D$ belonging to the arcs $A(h^l)$ of the full-twist $h^l$ and express the elements $A_{IQ}(h^l)$ in function of $x$, $y$, and $z$ as depicted in Figure \ref{fig16}. 

We will only prove the case $(i)$. The case $(ii)$ is proved by adapting the same argument we will use to prove the case $(i)$.

Assume $l=3$ and $| s_n | \geq 1$. In this case, the arcs $x$, $y$, and $z$ are depicted in the right of Figure \ref{fig16}. Suppose on the contrary that $IQ(L)$ is left-orderable and let $<$ be a left-order on the elements of $IQ(L)$. 

\begin{figure}[H]
\centering
\includegraphics[width=0.8\linewidth]{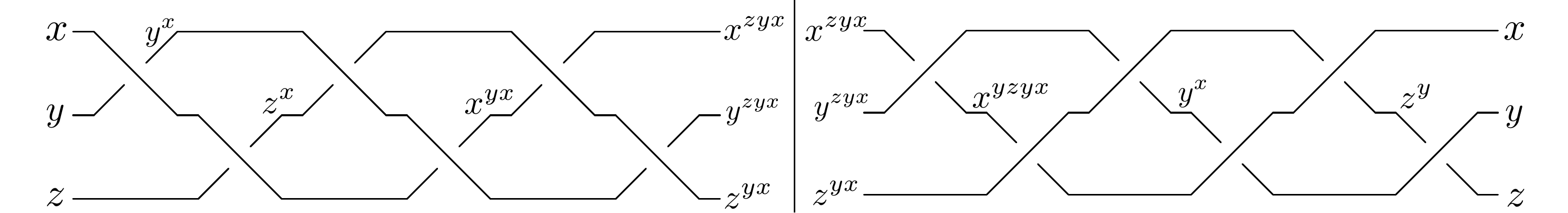}   
\caption{The elements of $A_{IQ}(h^l)$ written in function of the generators $x,y$ and $z$. The case where $l=1$ is depicted on the right and the case where $l=-1$ is depicted on the left.}
\label{fig16}
\end{figure} 

\item[\underline{Case 1}: $x=y$.] 
\item[\underline{Subcase 1}: $x=y=z$.] In this case, the elements of $A_{IQ}(D)$ depicted in Figure \ref{fig16} are all equal. By Proposition \ref{prop1}, we can show that all coarse relations in $\mathcal{IQ}(L,D)$ are trivial which is a contradiction by Theorem \ref{theo_1}. 
\item[\underline{Subcase 2}: $z<x$.] In this case, the orientation of the edges of the graph $\Gamma_{D,<}$ corresponding to the tangle parts $t_1$ and $s_1$ matches $S.D.T.1$. This imposes the orientation of the edges corresponding to all tangle parts $t_1,\hdots,t_n$ and $s_1,\hdots,s_n$. It is clear that the orientation of the chain of edges corresponding to $s_1,\hdots,s_n$ implies that $z<x^z$. But since $z<x$, then we get a contradiction by Lemma \ref{lem1}. 
\item[\underline{Subcase 3}: $x<z$.] In this case, the orientation of the edges of the graph $\Gamma_{D,<}$ corresponding to the tangle parts $t_1$ and $s_1$ matches $S.D.T.2$. This imposes the orientation of the edges corresponding to all tangle parts $t_1,\hdots,t_n$ and $s_1,\hdots,s_n$. It is clear that orientation of the chain of edges corresponding to $s_1,\hdots,s_n$ implies that $x^z<z$. But since $x<z$, then we get a contradiction by Lemma \ref{lem1}. 

We conclude that the case $x = y$ is impossible. Thus, $x \neq y$.
 
\item[\underline{Case 2:} $x<y$.] Note that by Proposition \ref{prop1}, this case is equivalent to $A_{IQ}(t_1) \subset \left[ x, \infty \right)_<$ , which determines the orientation of the edge of $\Gamma_{D,<}$ corresponding to $t_1$. In this case we have the following 
\begin{align}
x & < y \nonumber \\
\Rightarrow z^y*x & < z^y*y \nonumber  \\
\Leftrightarrow z^{yx} & < z. \label{eq13} 
\end{align}

Suppose that $A_{IQ}(s_1) \subset \left[ z, \infty \right)_<$. Then the orientation of the edges of $\Gamma_{D,<}$ corresponding to the tangle parts $t_1$ and $s_1$ are matching $S.D.T.1$. This imposes the orientation of the edges corresponding to all tangle parts $t_1,\hdots,t_n$ and $s_1,\hdots,s_n$. The orientation of the chain of edges corresponding to the tangle parts $s_1,\hdots,s_n$ implies that $z<z^{yx}$, which contradicts (\ref{eq13}). By Propositin \ref{prop1}, we conclude that  
\begin{equation}
A_{IQ}(s_1) \subset \left( \infty,z \right]_<.
\label{eq14}
\end{equation}

Let $\alpha$ be the single element of $A(t_1) \cap A(s_1)$. By the assumptions, Proposition \ref{prop1}, and (\ref{eq14}), we have the following 

\begin{align}
& x < y < \alpha \leq z \nonumber \\
\Rightarrow & y < z  \nonumber \\
\xLeftrightarrow{\mathit{Lemma \ref{lem1}}} & x^z*y < x^z*z \nonumber \\
\Leftrightarrow & x^{zy} < x \nonumber \\
\xLeftrightarrow{\mathit{Lemma \ref{lem1}}} & x < x^{zy}*x \nonumber \\
\Leftrightarrow & x < x^{zyx} \nonumber \\
\xLeftrightarrow{\mathit{Lemma \ref{lem1}}} & x^{yzyx}*z^{yx} < x^{yzyx}*y^{zyx}, \text{ because } x = x^{yzyx}*z^{yx} \text{ and } x^{zyx} = x^{yzyx}*y^{zyx} \nonumber \\
\Leftrightarrow & z^{yx} < y^{zyx} \nonumber \\
\xLeftrightarrow{\mathit{Proposition \ref{prop1}}} & A_{IQ}(s_n) \subset \left( \infty , y^{zyx} \right]_<, \text{ because } z^{yx} \in A_{IQ}(s_n) \text{ and } y^{zyx}=X_{s_n}. \label{eq15}
\end{align}

Note that (\ref{eq15}) determines the orientation of the edge of $\Gamma_{D,<}$ corresponding to $s_n$. Using (\ref{eq15}), we get the following

\begin{align}
& z^{yx} \leq y^{zyx} \nonumber \\
\xLeftrightarrow{\mathit{Lemma \ref{lem1}}} & y^{zyx}*z^{yx} \leq z^{yx} \nonumber \\
\Leftrightarrow & y^x \leq z^x*y^x, \text{ because } y^x = y^{zyx}*z^{yx} \text{ and } z^{yx}=z^x*y^x \nonumber \\
\xLeftrightarrow{\mathit{Lemma \ref{lem1}}} & z^x \leq y^x \nonumber \\
\xLeftrightarrow{\mathit{Lemma \ref{lem1}}} & z^{xzx}*z^x \leq z^{xzx}*y^x \nonumber \\
\Leftrightarrow & z \leq z^{yx}*x^{zyx}, \text{ because } z = z^{xzx}*z^x \text{ and } z^{yx}*x^{zyx}=z^{xzx}*y^x. \label{eq16} 
\end{align}

Let $\beta$ be the single element of $A(t_n) \cap A(s_n)$. Suppose that $A_{IQ}(t_n) \subset \left[ \beta , \infty \right)_<$, which determines the orientation of the edge corresponding to $t_n$. The resulting orientations of the edges corresponding to $t_n$ and $s_n$ matches $S.D.T.3$ which imposes the orientation of all the edges corresponding to $t_1,\hdots,t_n$ and $s_1,\hdots,s_n$. In particular, the resulting orientation of the edge corresponding to $t_1$  implies that $A_{IQ}(t_1) \subset \left( \infty , x \right]_<$. This contradicts the assumption $x < y$, because $y \in A_{IQ}(t_1)$, and so our supposition is false. Hence, we get that 

\begin{equation}
A_{IQ}(t_n) \subset \left( \infty , \beta \right]_<. \label{eq17}
\end{equation}

Suppose that $\beta \leq z^{yx}$. Since $x^{zyx} \in A_{IQ}(t_n)$, then (\ref{eq17}) implies the following 

\begin{align}
& x^{zyx} \leq \beta \leq z^{yx} \nonumber \\
\Rightarrow & x^{zyx} \leq z^{yx} \nonumber \\
\xLeftrightarrow{\mathit{Lemma \ref{lem1}}} & z^{yx}*x^{zyx} \leq x^{zyx} \leq z^{yx}. \label{eq18}
\end{align}

The results (\ref{eq16}) and (\ref{eq18}) combined imply that $z \leq z^{yx}$, which contradicts (\ref{eq13}). Thus, our supposition is false and we get that
\begin{equation}
z^{yx} < \beta \label{eq19}
\end{equation}

According to the assumptions, we have $| s_n | \geq 1$. This means that the fraction $\left[ a_1, \hdots, a_{2k+1} \right]$ of $s_n$ satisfies $a_1 \leq -1$. Diagrammatically, this means that if we enumerate the elementary subtangles of $s_n$ as $t_{n_1}, \hdots ,t_{n_{2k+1}}$ as in the right of Figure \ref{fig9} and the arcs of each $t_{n_i}$ as in Figure \ref{fig10}, then we have that $z^{yx} = x_{(n_1,1)}$. Then by Proposition \ref{prop1} and (\ref{eq15}), we get that $A_{IQ}(s_n) \setminus \left\lbrace y^{zyx} \right\rbrace \subset \left( \infty,z^{xy} \right]_<$, thus $\beta \leq z^{xy}$ which contradicts (\ref{eq19}). Hence, \underline{Case 2} is impossible.

\item[\underline{Case 3:} $y<x$.] We use an analogous argument as we did in \underline{Case 2} by inverting each one of the obtained inequalities and we show that this case is also impossible. We conclude thereby that if $l=3$ and $|s_n| \geq 1$, then $IQ(L)$ is not left-orderable.
\end{proof}

The set of quasi-alternating links appeared in the context of link homology as a natural generalization of alternating links. They were defined by Ozsv{\'a}th and Szab{\'o} in \cite{ozsvath} where it is shown that non-split alternating links are quasi-alternating. 
The set of quasi-alternating links is defined recursively as follows
\begin{mydef}
The set $\mathcal{Q}$ of \textbf{quasi-alternating links} is the smallest set of links satisfying the following properties:
\begin{enumerate}
\item The unknot belongs to $\mathcal{Q}$,
\item if $L$ is a link with a diagram $D$ containing a crossing $c$ such that 
\begin{enumerate}
\item for both smoothings of the diagram $D$ at the crossing $c$ denoted by $D^c_0$ and $D^c_\infty$ as in Figure \ref{fig13}, the links $L^{c}_{0}$ and $L^{c}_{\infty}$ represented respectively by the diagrams $D^{c}_{0}$ and $D^{c}_{\infty}$ are in $\mathcal{Q}$ and,
\item $\det(L)=\det(L^{c}_{0})+\det(L^{c}_{\infty}).$
\end{enumerate}
Then $L$ is in $\mathcal{Q}$. In this case we will say that $c$ is a
\textit{quasi-alternating crossing} of $D$ and that $D$ is quasi-alternating at $c$.
\end{enumerate}
\end{mydef}

\begin{figure}[H]
\centering
\includegraphics[width=0.4\linewidth]{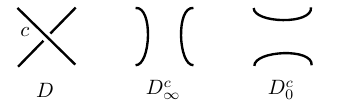}
\caption{The link diagram $D$ and its resolutions $D^{c}_{0}$ and $D^{c}_{\infty}$ at the crossing $c$.}
\label{fig13}
\end{figure}

A connected alternating link diagram is quasi-alternating at each of its \textit{non-nugatory crossings}, i.e., crossings that can not be eliminated by some finite sequence of Reidemeister moves. This is shown in \cite{ozsvath} and ensures that any non-split alternating link is quasi-alternating. 

 In \cite{baldwin2008heegaard}, Baldwin characterized quasi-alternating $3$-braid closures. According to this characterization, the only non-alternating and quasi-alternating $3$-braid closures are of the form $\sigma_1^{p_1}\sigma_2^{-q_1} \hdots \sigma_1^{p_n}\sigma_2^{-q_n} h^l$ where $p_i$ and $q_i$ are positive integers, and $l=\pm 3$. Thus, quasi-alternating and non-alternating $3$-braid closures are particular members of the family of links depicted in Figure \ref{fig12}. More precisely, these $3$-braid closures have the form $L \left( \left( -\frac{1}{p_1},-q_1 \right),\hdots,\left( -\frac{1}{p_n},-q_n \right),h^{\pm 3} \right)$. Consequently, we get the following result as a corollary of Theorem \ref{theo_2} and Theorem \ref{theo_4}.

\begin{cor}
The involutory quandle of any quasi-alternating $3$-braid closure is not left-orderable.
\label{cor_1}
\end{cor}

\section{Discussion and perspectives}
As mentioned in the introduction, the definition of the coarse presentation of the involutory quandle of a non-split link $L$ was inspired by the coarse Brunner's presentation of the fundamental group $\pi_1 \left( \Sigma_2(L) \right)$ introduced by Ito who then used it to give a non-left-orderability criterion of this fundamental group (\cite{ito2013non}, Theorem 3.11). This criterion was subsequently utilized to generate infinite families of links whose double branched covers have fundamental groups that are not left-orderable. Among these link families, several are obtained by modifications of quasi-alternating links which consist of the replacement of specific crossings with algebraic tangles (\cite{ito2013non}, Remark 2). Ito noted that the effectiveness of this technique stems from the fact that if a coarse Brunner's presentation argument allows to show that the criterion holds for a particular link $L$, then the same argument would allow to show that the criterion holds also for the links provided by the modifications of $L$. Besides, it is known that the quasi-alternating property is preserved by such modifications in numerous instances (\cite{champanerkar2009twisting},\cite{abchir2020generating}). These observations allowed Ito to highlight the usefulness of the criterion he introduced in detecting new families of quasi-alternating links whose double-branched covers possess non-left-orderable fundamental groups. Thus, Ito gives further support for the L-space conjecture (\cite{boyer2013spaces}). 

In Theorem \ref{theo_4}, we used a similar approach to Ito's by considering links that are obtained by modifying the non-alternating and quasi-alternating $3$-braid closures. In the early stages of our investigations, we checked that the criterion stated in Theorem \ref{theo_1} holds for these $3$-braid closures (Corollary \ref{cor_1}). Along the way, we noted, as Ito did in his context, that if a coarse presentation argument allows to show that the criterion holds for a given link, it would allow to show that the criterion holds also for the links obtained by replacing some crossings with specific rational tangles. Similarly to the discussion made in the previous paragraph, one can use the criterion stated in Theorem \ref{theo_1} to detect new families of quasi-alternating links whose involutory quandles are not left-orderable by modifying given ones for which the criterion holds. In particular, by Theorem \ref{theo_2}, the criterion holds for any alternating link which then can be modified to produce infinitely many non-alternating and quasi-alternating links (\cite{abchir2020generating}, Theorem 4.8).\\
  Furthermore, the class of augmented alternating links contains infinitely many quasi-alternating links (\cite{abchir2023infinite}, proof of Theorem 3.2). Among these links, there exist infinitely many Montesinos links (\cite{abchir2023infinite}, proof of Proposition 3.4). By Theorem \ref{theo_2}, the involutory quandles of these quasi-alternating links are not left-orderable.\\

Finally, this discussion leads us to state the following conjecture.

\begin{conj}
If $L$ is a quasi-alternating link, then its involutory quandle $IQ(L)$ is not left-orderable.
\label{conj_1}
\end{conj}

It is interesting to point out that there exist some quasi-alternating links for which the criterion in Theorem \ref{theo_1} does not allow to verify Conjecture \ref{conj_1}. For instance, the family of links depicted in Figure \ref{fig17} contains the quasi-alternating knot $9_{39}$ listed in the \textit{KnotInfo} tables (\cite{knotinfo}) and, as we shall see in Example \ref{empl2}, our argument fails to show that involutory quandles of the links in this family are not left-orderable. Before dealing with this example, we need the following general result concerning left-orderable involutory quandles.

\begin{lem}
  Let $Q$ be any left-orderable involutory quandle and let $a,b,c,d,e, \text{ and } f$ be any elements of $Q$.
  \begin{enumerate}
  \item If $a \in \left\lbrace | c,d | \right\rbrace$, then $\left\lbrace | a,b | \right\rbrace \subset \left\lbrace | c,b | \right\rbrace \cup \left\lbrace | d,b | \right\rbrace$.
  \item If $b \in \left\lbrace | e,f | \right\rbrace$, then $\left\lbrace | a,b | \right\rbrace \subset \left\lbrace | a,e | \right\rbrace \cup \left\lbrace | a,f | \right\rbrace$.
  \end{enumerate}
\label{lem6}
\end{lem}

The proof follows from a straightforward manipulation of sets of the form $\left\lbrace | x,y | \right\rbrace$. 

\begin{figure}[H]
\centering
\includegraphics[width=0.5\linewidth]{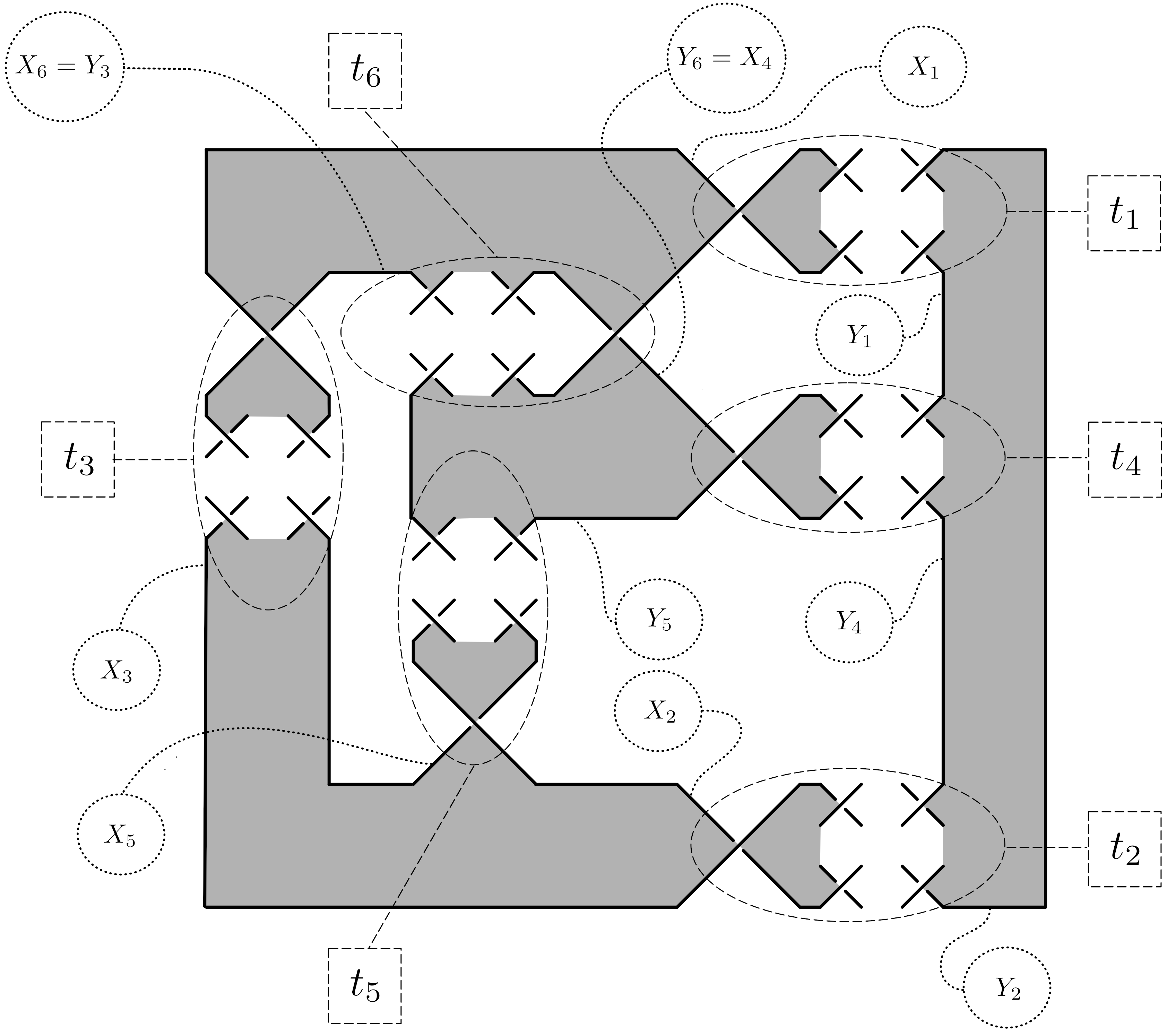}   
\caption{A tangle-strand decomposition of the diagram $D$ of Example \ref{empl2} with the labelling of some arcs.}
\label{fig17}
\end{figure} 

\begin{empl}
Let $D$ be the link diagram whose tangle-strand decomposition is depicted in Figure \ref{fig17}. Let $L$ be the corresponding link. To lighten the notations, for each tangle part $t_i$, we denote the arcs $X_{t_i}$ and $Y_{t_i}$ by $X_i$ and $Y_i$ respectively. It is clear from Figure \ref{fig17} that 
\begin{align}
 X_1 & \in A(t_3), X_2 \in A(t_5), X_3 \in A(t_2), X_5 \in A(t_3), Y_1 \in A(t_4), Y_2 \in A(t_1),  Y_4 \in A(t_2), Y_5 \in A(t_4), \notag \\ 
  X_4 & = Y_6 \text{ in } A(D), \text{ and } X_6 = Y_3 \text{ in } A(D).
\label{eq20}
\end{align}
Suppose that $IQ(L)$ is left-orderable. By (\ref{eq20}) and the coarse relations in $\mathcal{IQ}(L,D)$, we have the following 
\begin{equation}
\begin{array}{c|c}
  X_1 \in \left\lbrace | X_3,Y_3 | \right\rbrace. & Y_1 \in \left\lbrace | X_4,Y_4 | \right\rbrace. \\
  X_2 \in \left\lbrace | X_5,Y_5 | \right\rbrace. & Y_2 \in \left\lbrace | X_1,Y_1 | \right\rbrace. \\
  X_3 \in \left\lbrace | X_2,Y_2 | \right\rbrace. & Y_3 = X_6 \in \left\lbrace | X_6,Y_6 | \right\rbrace. \\
  X_4 = Y_6 \in \left\lbrace | X_6,Y_6 | \right\rbrace. & Y_4 \in \left\lbrace | X_2,Y_2 | \right\rbrace. \\
  X_5 \in \left\lbrace | X_3,Y_3 | \right\rbrace. & Y_5 \in \left\lbrace | X_4,Y_4 | \right\rbrace. \\
\end{array}
\label{eq21}
\end{equation}

By Lemma \ref{lem6} and (\ref{eq21}), we conclude that $A_{IQ}(D) = \displaystyle \bigcup_{1 \leq i \leq 6}\left\lbrace | X_i,Y_i | \right\rbrace \subset \left\lbrace | X_6,Y_6 | \right\rbrace$. Thus, if $<$ is any left-order on $IQ(L)$, then the candidates for the minimum and maximum of $A_{IQ}(D)$ with respect to $<$ are the generators $X_6$ and $Y_6$. In particular, If we could show that $X_6 = Y_6$ as generators of $IQ(L)$, the criterion in Theorem \ref{theo_1} would be met, hence $IQ(L)$ would be not left-orderable. Unfortunately, this equality has not been achieved. However, by using Proposition \ref{prop1}, one can show that if $IQ(L)$ is left-orderable and $<$ is any left-order on $IQ(L)$, then there exists a unique symbol $\diamond$ in the set $\left\lbrace <,> \right\rbrace$ such that the following inequalities hold
$$X_6 \diamond X_1 \diamond X_5 \diamond X_2 \diamond X_3 \diamond Y_4 \diamond Y_2 \diamond Y_1 \diamond Y_5 \diamond Y_6.$$
$$ X_2 \diamond \alpha \diamond \beta \diamond Y_2, \text{ where } \alpha = A(t_5) \cap A(t_6) \text{ and } \beta = A(t_1) \cap A(t_6).$$

\label{empl2}
\end{empl}
\nocite{*}
\bibliographystyle{plain}
\bibliography{on_left-orderability_of_involutory_quandles_of_links_abchir_sabak_06-11-2023}

\begin{thebibliography}{10}

\bibitem{abchir2020generating}
Hamid Abchir and Mohammed Sabak.
\newblock Generating links that are both quasi-alternating and almost
  alternating.
\newblock {\em Journal of Knot Theory and Its Ramifications}, 29(14):2050090,
  2020.

\bibitem{abchir2023infinite}
Hamid Abchir and Mohammed Sabak.
\newblock Infinite families of non-left-orderable $ l $-spaces.
\newblock {\em Osaka Journal of Mathematics}, 60(1):77--103, 2023.

\bibitem{adams}
Colin~C Adams.
\newblock Augmented alternating link complements are hyperbolic.
\newblock {\em Lowdimensional topology and Kleinian groups (Coventry/Durham,
  1984)}, 112:115--130, 1986.

\bibitem{baldwin2008heegaard}
John~A Baldwin.
\newblock Heegaard floer homology and genus one, one-boundary component open
  books.
\newblock {\em Journal of Topology}, 1(4):963--992, 2008.

\bibitem{bardakov2020embeddings}
Valeriy Bardakov and Timur Nasybullov.
\newblock Embeddings of quandles into groups.
\newblock {\em Journal of Algebra and its Applications}, 19(07):2050136, 2020.

\bibitem{bardakov2022zero}
Valeriy~G Bardakov, Inder Bir~S Passi, and Mahender Singh.
\newblock Zero-divisors and idempotents in quandle rings.
\newblock {\em Osaka Journal of Mathematics}, 59(3):611--637, 2022.

\bibitem{boyer2013spaces}
Steven Boyer, Cameron~McA Gordon, and Liam Watson.
\newblock On l-spaces and left-orderable fundamental groups.
\newblock {\em Mathematische Annalen}, 356(4):1213--1245, 2013.

\bibitem{boyer2005orderable}
Steven Boyer, Dale Rolfsen, and Bert Wiest.
\newblock Orderable 3-manifold groups.
\newblock In {\em Annales de l'institut Fourier}, volume~55, pages 243--288,
  2005.

\bibitem{brunner1997double}
AM~Brunner.
\newblock The double cover of s3 branched along a link.
\newblock {\em Journal of Knot Theory and Its Ramifications}, 6(05):599--619,
  1997.

\bibitem{champanerkar2009twisting}
Abhijit Champanerkar and Ilya Kofman.
\newblock Twisting quasi-alternating links.
\newblock {\em Proceedings of the American Mathematical Society},
  137(7):2451--2458, 2009.

\bibitem{conway}
John~H Conway.
\newblock An enumeration of knots and links, and some of their algebraic
  properties.
\newblock In {\em Computational Problems in Abstract Algebra (Proc. Conf.,
  Oxford, 1967)}, pages 329--358, 1970.

\bibitem{dabkowska2007compactness}
Malgorzata~A Dabkowska, Mieczyslaw~K Dabkowski, Valentina~S Harizanov, Jozef~H
  Przytycki, and Michael~A Veve.
\newblock Compactness of the space of left orders.
\newblock {\em Journal of Knot Theory and Its Ramifications}, 16(03):257--266,
  2007.

\bibitem{dhanwani2021dehn}
Neeraj~K Dhanwani, Hitesh Raundal, and Mahender Singh.
\newblock Dehn quandles of groups and orientable surfaces.
\newblock {\em Fundamenta Mathematicae}, 2023.

\bibitem{elhamdadi2015quandles}
Mohamed Elhamdadi and Sam Nelson.
\newblock {\em Quandles}, volume~74.
\newblock American Mathematical Soc., 2015.

\bibitem{fenn1992racks}
Roger Fenn and Colin Rourke.
\newblock Racks and links in codimension two.
\newblock {\em Journal of Knot theory and its Ramifications}, 1(04):343--406,
  1992.

\bibitem{goldman}
Jay~R Goldman and Louis~H Kauffman.
\newblock Rational tangles.
\newblock {\em Advances in Applied Mathematics}, 18(3):300--332, 1997.

\bibitem{ha2018algorithmic}
Trang Ha.
\newblock {\em On algorithmic properties of computable magmas}.
\newblock PhD thesis, The George Washington University, 2018.

\bibitem{ha2018orders}
Trang Ha and Valentina Harizanov.
\newblock Orders on magmas and computability theory.
\newblock {\em Journal of Knot Theory and Its Ramifications}, 27(07):1841001,
  2018.

\bibitem{hoste2017links}
Jim Hoste and Patrick Shanahan.
\newblock Links with finite n--quandles.
\newblock {\em Algebraic \& Geometric Topology}, 17(5):2807--2823, 2017.

\bibitem{hoste2017involutory}
Jim Hoste and Patrick~D Shanahan.
\newblock Involutory quandles of (2, 2, r)-montesinos links.
\newblock {\em Journal of Knot Theory and Its Ramifications}, 26(03):1741003,
  2017.

\bibitem{ito2013non}
Tetsuya Ito.
\newblock Non-left-orderable double branched coverings.
\newblock {\em Algebraic \& Geometric Topology}, 13(4):1937--1965, 2013.

\bibitem{joyce1982classifying}
David Joyce.
\newblock A classifying invariant of knots, the knot quandle.
\newblock {\em Journal of Pure and Applied Algebra}, 23(1):37--65, 1982.

\bibitem{joyce1979algebraic}
David~Edward Joyce.
\newblock An algebraic approach to symmetry with applications to knot theory.
\newblock 1979.

\bibitem{kauffman2}
Louis~H Kauffman and Sofia Lambropoulou.
\newblock On the classification of rational tangles.
\newblock {\em Advances in Applied Mathematics}, 33(2):199--237, 2004.

\bibitem{knotinfo}
Charles Livingston and Allison~H. Moore.
\newblock Knotinfo: Table of knot invariants.
\newblock URL: \url{knotinfo.math.indiana.edu}, Current Month Current Year.

\bibitem{matveev1984distributive}
Sergei~Vladimirovich Matveev.
\newblock Distributive groupoids in knot theory.
\newblock {\em Math. USSR Sbornik}, 47(1):73--83, 1984.

\bibitem{mellor2022finite}
Blake Mellor.
\newblock Finite involutory quandles of two-bridge links with an axis.
\newblock {\em Journal of Knot Theory and Its Ramifications}, 31(02):2250009,
  2022.

\bibitem{ozsvath}
Peter Ozsv{\'a}th and Zolt{\'a}n Szab{\'o}.
\newblock On the heegaard floer homology of branched double-covers.
\newblock {\em Advances in Mathematics}, 194(1):1--33, 2005.

\bibitem{raundal}
Hitesh Raundal, Mahender Singh, and Manpreet Singh.
\newblock Orderability of link quandles.
\newblock {\em Proceedings of the Edinburgh Mathematical Society},
  64(3):620--649, 2021.

\bibitem{sikora2004topology}
Adam~S Sikora.
\newblock Topology on the spaces of orderings of groups.
\newblock {\em Bulletin of the London Mathematical Society}, 36(4):519--526,
  2004.

\bibitem{whitney2009non}
Hassler Whitney.
\newblock Non-separable and planar graphs.
\newblock In {\em Classic Papers in Combinatorics}, pages 25--48. Springer,
  2009.

\bibitem{winker1984quandles}
Steven~Karl Winker.
\newblock {\em Quandles, knot invariants, and the n-fold branched cover}.
\newblock University of Illinois at Chicago, 1984.

\end{thebibliography}
\end{document}